\documentclass[10pt]{amsart}
\usepackage{amssymb, accents, tikz, tikz-cd, tabto, stmaryrd, mathtools, enumitem, comment, xcolor, graphicx, mathrsfs} 

\usepackage{mathbbol}

\usepackage{scalerel}\newcommand{\bigboxtimes}{\mathop{\scalerel*{\boxtimes}{\sum}}}

\usepackage[style = alphabetic, doi = false, url=true, isbn = false]{biblatex}
\usepackage[hyperfootnotes=false]{hyperref}
\usepackage{eucal}
\usepackage{xurl}
\hypersetup{breaklinks=true}
\renewbibmacro{in:}{}
\theoremstyle{definition}
\newtheorem{theorem}{Theorem}[section]

\newtheorem{remark}[theorem]{Remark}

\newtheorem{example}[theorem]{Example}

\newtheorem{proposition}[theorem]{Proposition}
\newtheorem{lemma}[theorem]{Lemma}
\newtheorem{corollary}[theorem]{Corollary}

\numberwithin{equation}{section}

\allowdisplaybreaks

\title{Comparing dg category models for path spaces via $A_\infty$-functors}

  \author[M.~Rivera]{Manuel~Rivera}
  \address{Manuel Rivera,
  Department of Mathematics, Purdue University, 150 N University St, West Lafayette, Indiana, 47907 }
  \email{manuelr@purdue.edu}

  \author[Y.~Wang]{Yi~Wang}
  \address{Yi Wang,
  Department of Mathematics, Purdue University, 150 N University St, West Lafayette, Indiana, 47907 }
  \email{wang6206@purdue.edu}

\begin{document}
\begin{abstract}
We construct a many-object dual version of Chen’s iterated integral map. For any topological space $X$, the construction takes the form of an $A_\infty$-functor between two dg categories whose objects are the points of $X$: the domain has as morphisms the singular (cubical) chains on the space of (Moore) paths in $X$ and the codomain has morphisms arising by totalizing a cosimplicial chain complex determined by the dg coalgebra of singular (simplicial) chains in $X$. When $X$ is simply connected, we show this construction defines a homotopy inverse to a classical map of Adams, which sends ordered sequences of singular simplices in $X$ linked by shared vertices to cubes of paths in $X$. When $X$ is not necessarily simply connected, following an idea of Irie, we incorporate the fundamental groupoid of $X$ into the construction and deduce analogous results. Along the way, we provide an elementary and new proof of the fact that the (direct-sum) cobar construction of the chains in $X$, suitably interpreted, models the dg category of paths in $X$, an extension of Adams’s cobar theorem established by Rivera-Zeinalian using different methods.
\end{abstract}

\maketitle
\section{Introduction}\label{section: introduction}
Denote by 
\[ \mathsf{P} \colon \mathsf{Top} \to \mathsf{Cat}_{\mathsf{Top}}\]
the functor that associates to any topological space\footnote{Throughout this paper, topological spaces are always assumed to be locally path-connected and  semilocally simply connected.} $X$ the topologically enriched category $\mathsf{P}X$ defined as follows. The objects of $\mathsf{P}X$ are the points of $X$ and, for any two $a,b \in X$, the morphism space $\mathsf{P}X(a,b)$ is the set of all pairs $(\gamma, T)$, where $T \in \mathbb{R}_{\geq 0}$ and $\gamma \colon [0,T] \to X$ is a continuous path with $\gamma(0)=a$ and $\gamma(T)=b$, equipped with the compact-open topology. The composition rule on $\mathsf{P}X$ is defined by concatenating paths and adding the corresponding parameters. The identity morphism at an object $b \in X$ is given by the constant path $\gamma \colon \{0\} =[0,0] \to X$, $\gamma(0)=b$. The topological monoid $\mathsf{P}X(b,b)$ is the space $\Omega_bX$ of Moore loops in $X$ based at $b$. 

For a fixed commutative ring with unit $R$, denote by 
\[\mathsf{C}^{\square} \colon \mathsf{Cat}_{\mathsf{Top}} \to \mathsf{dgCat}_R\] 
the functor that associates to any topologically enriched category the differential graded (dg) $R$-category obtained by applying the normalized cubical singular chains with coefficients in $R$ at the level of spaces of morphisms. Similarly, denote by $\mathsf{C}^{\mathbb{\Delta}}$ the functor that applies normalized simplicial singular chains instead. 
The main purpose of this article is to compare the functor 
\begin{equation}\label{equation: cube compose path}
    \mathsf{C}^{\square} \circ \mathsf{P} \colon \mathsf{Top} \to \mathsf{dgCat}_R 
\end{equation}
and the following different versions of the \textit{cobar construction}.

\begin{enumerate}[leftmargin=*, label=(\arabic*), ref=(\arabic*)]

\item \label{item: version of cobar - classical} \textit{The classical cobar construction.} This is a functor
\[
\mathsf{Cobar}\colon \mathsf{dgCoalg}_R\to\mathsf{dgAlg}_R
\]
which takes a dg  coaugmented (coassociative) $R$-coalgebra $C$ as input and produces a dg augmented (associative) $R$-algebra $\mathsf{Cobar}(C)$. The underlying graded algebra of $\mathsf{Cobar}(C)$ is the free algebra
$T(s^{-1}\overline{C})= \bigoplus_{n=0}^{\infty} (s^{-1}\overline{C})^{\otimes n}$
where $\overline{C}$ is the cokernel of the coaugmentation of $C$ and $s^{-1}$ denotes the shift functor, i.e., $(s^{-1}\overline{C})_i= \overline{C}_{i+1}$. The differential of $\mathsf{Cobar}(C)$ is induced by the sum $\partial + \Delta$, where $\partial \colon C \to C$ is the differential of $C$ and $\Delta \colon C \to C \otimes C$ is the coassociative coproduct of $C$. The functor $\mathsf{Cobar}$ is a left adjoint with right adjoint being the classical bar construction $\mathsf{Bar}$. The main example we apply this construction to is the dg coaugmented coalgebra of singular chains on a pointed space $(X,b)$. 

\item \label{item: version of cobar - many} (See \cite{R, HL}). \textit{A ``many object" version of the cobar construction.} This is a functor  
\[
\mathsf{Cobar}^{\boxtimes}\colon \mathsf{cCoalg}_R\to \mathsf{dgCat}_R
\]
that takes a \textit{categorical $R$-coalgebra} $\mathcal{C}$ to a dg category $\mathsf{Cobar}^{\boxtimes}(\mathcal{C})$. Three features of a categorical coalgebra $\mathcal{C}$ are (i) $\mathcal{C}$ is a non-negatively graded $R$-coalgebra such that $\mathcal{C}_0=R[\mathcal{S}]$ for some set $\mathcal{S}$, (ii) the ``differential'' $\partial$ on $\mathcal{C}$ does not square to zero but its failure is controlled by a ``curvature" term, and (iii) $\mathcal{C}$ has a compatible $\mathcal{C}_0$-bicomodule structure. The set of objects in $\mathsf{Cobar}^{\boxtimes}(\mathcal{C})$ is $\mathcal{S}$. For $a,b\in \mathcal{S}$, the definition of $\mathsf{Cobar}^{\boxtimes}(\mathcal{C})(a,b)$ is recalled in section \ref{section: different versions cobar}; here we just point that the generators are ordered sequences of elements of $\mathcal{C}$ ``connecting" $a$ and $b$ (expressed as monomials over the cotensor product $\boxtimes$ over $\mathcal{C}_0$) and the differential on $\mathsf{Cobar}^{\boxtimes}(\mathcal{C})(a,b)$ is induced by $\partial+\Delta+h$ where $h$ is the curvature of $\mathcal{C}$. A natural example of a categorical coalgebra may be obtained for any topological space $X$ with $\mathcal{S}$ being the underlying set of $X$ by slightly modifying the dg coalgebra of singular chains on $X$ giving a functor \[\mathcal{C}^{\mathbb{\Delta}}\colon \mathsf{Top}\to\mathsf{cCoalg}_R.\] For simplicity, we write $\mathcal{C}^{\mathbb{\Delta}}(X)$ as $\mathcal{C}(X)$.

\item \label{item: version of cobar - many prod} \textit{A cosimplicial totalized ``many-object" version of the cobar construction.} This is a functor
\[
\mathsf{Cobar}^{\prod}\colon \mathsf{dgCoalg}_R^{\text{obj}} \to \mathsf{dgCat}_R
\]
where the input is a dg $R$-coalgebra $C$ such that $C_0=R[S]$ for some set of cycles $S$, which we call ``objects". The set of objects in $\mathsf{Cobar}^{\prod}(C)$ is $S$, and for $a,b\in S$, $\mathsf{Cobar}^{\prod}(C)(a,b)$ is defined as the totalization of certain cosimplicial chain complex defined by the coalgebra structure of $C$ (and hence the direct product symbol in the notation); see Section \ref{section: different versions cobar} for more details. An example of such a $C$ is obtained from a topological space $X$ with $S$ being the underlying set of $X$ and $C$ being the dg coalgebra of singular chains on $X$; this gives a functor \[C^{\mathbb{\Delta}}\colon \mathsf{Top}\to\mathsf{dgCoalg}_R^{\text{obj}},\] and, for simplicity, we may write $C^{\mathbb{\Delta}}(X)$ as $C(X)$.
\end{enumerate}

The ``many object'' versions \ref{item: version of cobar - many}\ref{item: version of cobar - many prod} of the cobar construction lead to functors
\begin{align}\label{equation: functor cobar box top}
\mathsf{Cobar}^{\boxtimes}\circ\mathcal{C}^{\mathbb{\Delta}}\colon \mathsf{Top}\to\mathsf{dgCat}_R,\\
\label{equation: functor cobar otimes top}
\mathsf{Cobar}^{\prod}\circ C^{\mathbb{\Delta}}\colon \mathsf{Top}\to\mathsf{dgCat}_R.
\end{align}
In this article, we establish an explicit relationship between the functors \eqref{equation: cube compose path}, \eqref{equation: functor cobar box top}, and \eqref{equation: functor cobar otimes top}.

First, a classical construction of Adams' gives a natural transformation
\[
\mathcal{A}\colon \mathsf{Cobar}^{\boxtimes}\circ\mathcal{C}^{\mathbb{\Delta}}\Rightarrow \mathsf{C}^{\square}\circ\mathsf{P}.
\]
The idea behind the construction of $\mathcal{A}$ is to associate to any ordered sequence of simplices in $X$ linked by a shared vertex (a ``necklace" in $X$) a suitable cube of paths connecting the first and last vertices in the sequence.
\begin{remark}
    Adams \cite{Adams} originally worked with $\mathsf{Cobar}(C^1(X,b))$, where $b \in X$ is a base point and $C^1(X,b)$ is the dg coaugmented coalgebra of 1-reduced singular chains on $(X,b)$. This construction naturally extends to $\mathsf{Cobar}^{\boxtimes}(\mathcal{C}^\mathbb{\Delta}(X))$. 
\end{remark}

Next, for any topological space $X$, a variation (or ``pre-dual" version) of Chen's iterated integral map gives natural chain maps
\begin{equation}\label{equation: I_1}
It_X \colon \mathsf{C}^{\square}(\mathsf{P}X)(a,b) \to \mathsf{Cobar}^{\prod}(C(X))(a,b),
\end{equation}
for all $a,b \in X$, induced by evaluating each path in $\mathsf{P}X(a,b)$ at all ordered sequences of $n$ marked points (which are determined by points in the $n$-simplex), for all $n\ge 0$;  see Section \ref{Section: A formal dual of Chen's iterated integral map} for more details.

The maps \ref{equation: I_1} all together do not define a (strict) functor of dg categories $\mathsf{C}^{\square}(\mathsf{P}X) \to  \mathsf{Cobar}^{\prod}(C(X))$ since compositions are not preserved. However, we have the following statement, which is the first main result of this article.

\begin{theorem}\label{theorem: A infinity map I}
The natural maps in \eqref{equation: I_1}
can be extended to a natural $A_\infty$-transformation 
\[
\mathcal{I} = \{\mathcal{I}_n\}_{n \ge 1}\colon \mathsf{C}^{\square}\circ\mathsf{P} \Rightarrow \mathsf{Cobar}^{\prod}\circ C^{\mathbb{\Delta}},
\]
in the sense that for any topological space $X$,
\[
\mathcal{I}_X=\{\mathcal{I}_{X,n}\}_{n\ge1}\colon \mathsf{C}^{\square}(\mathsf{P}X)\to \mathsf{Cobar}^{\prod}(C(X))
\]
is an $A_\infty$-functor with $\mathcal{I}_{X,1} = It_X$, and the family $\{\mathcal{I}_X\}_{X \in \mathsf{Top}}$ is natural in $X$.
\end{theorem}

\begin{remark}
    Chen \cite{Chen} originally worked with cochains (differential forms) on what he called differentiable spaces, which generalize smooth manifolds. In contrast, we work with singular chains on topological spaces, leading to the map $\mathcal{I}_1$ which is formally dual to Chen's. 
    The higher components $\mathcal{I}_n$ ($n>1$) can be viewed geometrically as arising from subdivisions of simplices compatible with path concatenation. These higher homotopies disappear after pairing with differential forms via integration and are therefore not seen in Chen's framework (see Remark~\ref{remark: I formal dual to Chen} and the subsequent discussion for details).

\end{remark}

A natural transformation between two functors $\mathsf{Top}\to \mathsf{dgCat}_R$ is a natural $A_\infty$-transformation with vanishing higher components. Hence, given the natural $A_\infty$-transformations $\mathcal{A}$ and $\mathcal{I}$, their composition $\mathcal{I}\circ\mathcal{A}$ is a natural $A_\infty$-transformation
\[
\mathcal{F}=\{\mathcal{F}_n\}_{n\geq1}=\{\mathcal{I}_n\circ\mathcal{A}\}_{n\geq1}\colon \mathsf{Cobar}^{\boxtimes}\circ\mathcal{C}^{\mathbb{\Delta}}\Rightarrow \mathsf{Cobar}^{\prod}\circ C^{\mathbb{\Delta}}.
\]
On the other hand, there is a straightforwardly defined natural transformation \[ \mathcal{G}\colon \mathsf{Cobar}^{\boxtimes}\circ\mathcal{C}^{\mathbb{\Delta}}\Rightarrow \mathsf{Cobar}^{\prod}\circ C^{\mathbb{\Delta}} \] that, roughly speaking, is induced by the inclusion of a direct sum into a direct product (modulo the subtle difference between $\mathcal{C}^{\mathbb{\Delta}}$ and $C^{\mathbb{\Delta}}$). The following is the second main result of the present article.

\begin{theorem}\label{theorem: A infinity homotopy}

The natural $A_\infty$-transformations $\mathcal{F}$ and $\mathcal{G}$ are $A_\infty$-homotopic, i.e., for any topological space $X$, there exists a natural $A_\infty$-homotopy $\mathcal{H}_X$ between the $A_\infty$-functors $\mathcal{F}_X$ and $\mathcal{G}_X$.
\end{theorem}

\begin{remark}\label{remark: Chen Adams inverse simply connected}
    Suppose $X$ is simply connected, and fix a basepoint $b\in X$. If one restricts to the dg coalgebra of 1-reduced singular chains $ C^1(X,b)$, then $\mathcal{G}_X$ identifies $\mathsf{Cobar}^{\boxtimes}$ with the conormalization of $\mathsf{Cobar}^{\prod}$.
    Thus, Theorem \ref{theorem: A infinity homotopy} implies:
    \begin{quote}
    \emph{Chen's iterated integral map, suitably extended and reinterpreted, is a left $A_\infty$-homotopy inverse of Adams' map for simply connected topological spaces.}
    \end{quote}
    See Section \ref{subsection: Chen Adams inverse simply connected} for a more detailed statement.
\end{remark}

Recall that a functor $\mathsf{C}_{1}\to \mathsf{C}_{2}$ between dg categories is called a 
\emph{quasi-equivalence} if it induces quasi-isomorphisms on all morphism complexes and an 
equivalence $H_{0}(\mathsf{C}_{1})\simeq H_{0}(\mathsf{C}_{2})$. 
The same notion applies to $A_{\infty}$-categories and $A_{\infty}$-functors, 
where it is called an \emph{$A_{\infty}$-quasi-equivalence}.  Building on Remark~\ref{remark: Chen Adams inverse simply connected}, we then obtain:

\begin{corollary}\label{corollary: Adams equiv iff Chen equiv}
    For any simply connected topological space $X$, the functor \[\mathcal{A}_X \colon \mathsf{Cobar}^{\boxtimes}(\mathcal{C}(X))\to \mathsf{C}^{\square}(\mathsf{P}X)\] induced by Adams' map is a quasi-equivalence if and only if the $A_\infty$-functor \[\mathcal{I}_X \colon \mathsf{C}^{\square}(\mathsf{P}X)\to\mathsf{Cobar}^{\prod}(C(X))\] motivated by Chen's iterated integral map is an $A_\infty$-quasi-equivalence.
\end{corollary}

Recall that Adams \cite{Adams} constructed the map $\mathsf{Cobar}(C^1(X,b)) \to C^\square(\Omega_b X)$ and proved it is a quasi-isomorphism for simply connected $X$ by comparing the Serre spectral sequences associated with the path fibration. Chen \cite{Chen} proved that his original iterated integral map, which takes an ordered sequence of differential forms on a manifold $M$ as input and produces a cochain in $\Omega_bM$ as output, is a quasi-isomorphism for simply connected $M$ by pairing with Adams' cobar construction and using spectral sequence arguments. Thus, Corollary~\ref{corollary: Adams equiv iff Chen equiv} highlights the structural feature underlying Chen’s proof and provides a conceptual explanation for why that argument works.

An adaptation of Adams’ original argument yields that, for any simply connected space $X$, the dg functor $\mathcal{A}_X \colon \mathsf{Cobar}^{\boxtimes}(\mathcal{C}(X))\to \mathsf{C}^{\square}(\mathsf{P}X)$ is a quasi equivalence (in fact, we prove a stronger statement in this article, see Theorem~\ref{theorem: Adams} below). Together with Corollary~\ref{corollary: Adams equiv iff Chen equiv}, this immediately implies the following.

\begin{corollary}\label{corollary: Chen is a quasi-equiv}
 If $X$ is simply connected, the $A_\infty$-functor \[\mathcal{I}_X \colon \mathsf{C}^{\square}(\mathsf{P}X)\to\mathsf{Cobar}^{\prod}(C(X))\] is an $A_\infty$-quasi-equivalence.
\end{corollary}

A natural question is whether Corollaries~\ref{corollary: Adams equiv iff Chen equiv} and ~\ref{corollary: Chen is a quasi-equiv} extend to possibly non-simply connected spaces. One may begin to address this by examining whether Adams' map and Chen's iterated integral map remain quasi-isomorphisms when $X$ is not simply connected. In recent years, Adams' theorem has been shown to extend to arbitrary topological spaces \cite{RZ_cubical}. In fact, it turns out that $\mathcal{A}_X$ is a quasi-equivalence even when $X$ is a non-simply connected space. We give an elementary proof (independent of all the above results) of this extension of Adams’ theorem by lifting necklaces to the universal cover and using Adams' original proof for simply connected spaces.

\begin{theorem}\label{theorem: Adams}
    For any topological space $X$, the functor 
    \[ \mathcal{A}_X\colon \mathsf{Cobar}^{\boxtimes}(\mathcal{C}(X))\rightarrow \mathsf{C}^{\square}(\mathsf{P}X)\]
    is a quasi-equivalence of dg categories.
\end{theorem}

\begin{corollary}[\cite{RZ_cubical,rivera2022adams}] \label{corollary: pointed Adams}
For any pointed topological space $(X,b)$, Adams' map \[\mathsf{Cobar}(C^0(X,b)) \rightarrow C^{\square}(\Omega_bX)\] is a quasi-isomorphism of dg algebras.
\end{corollary}

In contrast, Chen's iterated integral map is not a quasi-isomorphism unless additional assumptions are imposed on the fundamental group. Similarly, the map $\mathcal{I}_X$ is generally not a quasi-equivalence for an arbitrary topological space $X$. To address this issue in the non-simply connected case, it is necessary to modify the functor $\mathsf{Cobar}^{\prod} \circ C^{\mathbb{\Delta}}$ and, accordingly, the natural $A_\infty$-transformation $\mathcal{I}$. To this end, we incorporate the fundamental groupoid into the construction, following an idea of Irie \cite{Irie, Wang2023thesis}.
The modification of $\mathsf{Cobar}^{\prod}\circ C^{\mathbb{\Delta}}$ defines a functor
\[
\mathsf{C^P}\colon \mathsf{Top}\to \mathsf{dgCat}_R,
\]
and there are natural chain maps
\[
\widetilde{It}_X \colon \mathsf{C}^{\square}(\mathsf{P}X)(a,b)\to \mathsf{C^P}(X)(a,b),\qquad X\in\mathsf{Top},\quad a,b\in X
\]
induced by evaluation maps on path spaces.
There is also a natural transformation
\[
\widetilde{\mathcal{G}}\colon\mathsf{Cobar}^{\boxtimes}\circ \mathcal{C}^{\mathbb{\Delta}}\Rightarrow \mathsf{C^P}
\]
defined analogously to $\mathcal{G}$.
If $X$ is simply connected, then $\mathsf{C^P}(X)$, $\widetilde{It}_X$, and $\widetilde{\mathcal{G}}_X$ coincide with $\mathsf{Cobar}^{\prod}(C(X))$, $It_X$, and $\mathcal{G}_X$, respectively.

We prove the following analogue of Theorem~\ref{theorem: A infinity map I} and Theorem~\ref{theorem: A infinity homotopy} as the fourth main result of this article.

\begin{theorem}\label{theorem: A infinity map and homotopy with pi_1}
    The collection of natural chain maps $\{\widetilde{It}_X\}_{X\in\mathsf{Top}}$ extends to a natural $A_\infty$-transformation
    \[ \widetilde{\mathcal{I}}\colon\mathsf{C}^{\square}\circ \mathsf{P}\Rightarrow\mathsf{C^P}.
    \]
    Moreover, there exists an $A_\infty$-homotopy $\widetilde{\mathcal{H}}$ between the natural $A_\infty$-transformations 
    \[\widetilde{\mathcal{F}}=\widetilde{\mathcal{I}}\circ\mathcal{A}\text{ and }\widetilde{\mathcal{G}} \;\colon\; \mathsf{Cobar}^{\boxtimes}\circ \mathcal{C}^{\mathbb{\Delta}}\Rightarrow \mathsf{C^P}.
    \]
\end{theorem}

For any topological space $X$, an analogue of Corollary~\ref{corollary: Adams equiv iff Chen equiv} holds where
$\mathcal{F}_X$, $\mathcal{G}_X$, and $\mathcal{I}_X$ are replaced by 
$\widetilde{\mathcal{F}}_X$, $\widetilde{\mathcal{G}}_X$, and $\widetilde{\mathcal{I}}_X$, respectively. 
Together with Theorem~\ref{theorem: Adams}, this shows that $\widetilde{\mathcal{I}}_X$ is a quasi-equivalence for any topological space $X$. In this sense, $\widetilde{\mathcal{I}}_X$ provides a formal and ``correct'' extension of Chen's iterated integral map, guaranteeing its quasi-equivalence property beyond the simply connected case. See Section~\ref{section: non simply connected results} for further details.

\subsection*{Acknowledgment}
We thank Kei Irie for sharing ideas on constructing chain models of loop spaces, which contributed to the development of the present work. The first author acknowledges the excellent working conditions at Northwestern University, where part of this research was carried out and where valuable discussions with Ezra Getzler, Dennis Sullivan, Boris Tsygan, and Mahmoud Zeinalian took place during the workshop \textit{Homotopical Algebra in Geometry, Topology, and Physics}. Both the first author and the workshop received support from the NSF grant DMS-245405.

\section{Main Constructions}

In this section, we give further details on the cobar constructions $\mathsf{Cobar}^{\boxtimes}$, $\mathsf{Cobar}^{\prod}$, the natural transformations $\mathcal{A}$, $\mathcal{G}$ and the natural maps $\{It_X\}_{X\in\mathsf{Top}}$ introduced in Section \ref{section: introduction}.

\label{section: main constructions}
\subsection{Different versions of the cobar construction} In this subsection, we give more details on the non-classical versions of the cobar construction.
\label{section: different versions cobar}

\subsubsection{The ``many object" version \ref{item: version of cobar - many}} Given a set $\mathcal{S}$, denote by $R[\mathcal{S}]$ the cocommutative counital coalgebra given by the free $R$-module on $\mathcal{S}$ equipped with the coproduct determined by $b \mapsto b\otimes b$ for all $b \in \mathcal{S}$. We shall define a version of the cobar construction which takes as an input the following notion introduced in \cite{R}.

A \textit{categorical $R$-coalgebra} consists of a tuple $(\mathcal{C}, \partial, \Delta,  h)$ where
\begin{enumerate}
    \item $(\mathcal{C},\Delta)$ is a non-negatively graded coassociative counital coalgebra that is flat as an $R$-module
    \item $\partial:\mathcal{C} \rightarrow \mathcal{C}$ is a degree $-1$ coderivation of the coproduct $\Delta$
    \item $h: \mathcal{C} \rightarrow R$ is a linear map of degree $-2$ satisfying $h\circ \partial =0$ and \[\partial^2 = (h \otimes \text{id} - \text{id} \otimes h) \circ \Delta ,\] i.e. $h$ is a \textit{curvature} for $(\mathcal{C},\partial,
    \Delta)$
    \item The set \[\mathcal{S}(\mathcal{C}) = \{ x \in \mathcal{C} : \Delta(x)=x\otimes x \text{ and } \varepsilon(x)=1 \},\]
   where $\varepsilon \colon \mathcal{C} \to R$ denotes the counit, is non-empty and the natural inclusion map $\mathcal{S}(\mathcal{C})\hookrightarrow \mathcal{C}_0$ induces an isomorphism of coalgebras
    \[ R[\mathcal{S}(\mathcal{C})]\cong \mathcal{C}_0.\]
    \item The natural projection $\epsilon: \mathcal{C}\rightarrow \mathcal{C}_0$ satisfies $ \epsilon \circ \partial= 0 $
\end{enumerate}

Any categorical coalgebra $(\mathcal{C},\partial,\Delta,h)$ has a natural $\mathcal{C}_0$-bicomodule structure with structure maps given by
\[ \rho_r \colon \mathcal{C} \xrightarrow{\Delta} \mathcal{C} \otimes \mathcal{C} \xrightarrow{\text{id} \otimes \epsilon} \mathcal{C} \otimes \mathcal{C}_0\]
and
\[ \rho_l \colon \mathcal{C} \xrightarrow{\Delta} \mathcal{C} \otimes \mathcal{C} \xrightarrow{\epsilon \otimes \text{id} } \mathcal{C}_0 \otimes \mathcal{C}.\]

Categorical coalgebras form a category with the following notion of morphism. A morphism from $(\mathcal{C},\partial,\Delta,h)$ to $(\mathcal{C}',\partial',\Delta', h')$ consists of a pair $f=(f_0,f_1)$ where
\begin{enumerate}
    \item $f_0:(\mathcal{C},\Delta)\rightarrow (\mathcal{C}',\Delta')$ is a morphism of graded coalgebras 
    \item $f_1: \mathcal{C} \rightarrow \mathcal{C}_0^\prime$ is a $\mathcal{C}_0^{\prime}$-bicomodule map of degree $-1$
\end{enumerate}
satisfying
\begin{align*}
    & f_0\circ \partial = \partial' \circ f_0 + (\overline{f}_1 \otimes f_0)\circ(\Delta - \Delta^{\text{op}}) \text{ and}\\
    & h'\circ f_0 = h+ \overline{f}_1 \circ \partial + (\overline{f}_1\otimes \overline{f}_1)\circ \Delta , 
\end{align*}
where $\overline{f}_1 = \varepsilon' \circ f_1$ and $\varepsilon'$ is the counit of $\mathcal{C}'$. The composition of two morphisms is defined by
\begin{equation}
    (g_0,g_1)\circ(f_0,f_1)= (g_0 \circ f_0, g_1\circ f_0 + g_0 \circ f_1).
\end{equation}
We denote by $\mathsf{cCoalg}_R$ the category of categorical coalgebras.

Given a categorical coalgebra $\mathcal{C}$, the dg category $\mathsf{Cobar}^{\boxtimes}(\mathcal{C})$ is defined as follows. The objects are the elements of $\mathcal{S}$. Denote by $\boxtimes$ the cotensor product over $\mathcal{C}_0$ and by $s^{-1}$ the degree shift functor by $-1$. For any two $a,b \in \mathcal{S}$ we define
\[\mathsf{Cobar}^{\boxtimes}(\mathcal{C})(a,b)=
\bigoplus_{n=0}^{\infty} R_a \boxtimes (s^{-1}\overline{\mathcal{C}})^{\boxtimes n} \boxtimes R_b, \]
where, for any $a \in \mathcal{S}$, $R_a$ is $R$ equipped with the $R[\mathcal{S}]$-bicomodule determined by the inclusion $\{a\} \hookrightarrow \mathcal{S}$, $\overline{\mathcal{C}}= \bigoplus_{i=1}^{\infty} \mathcal{C}_i$, and $(s^{-1}\overline{\mathcal{C}})^{\boxtimes 0}= \mathcal{C}_0$. The identity morphism at $a \in \mathcal{S}$ corresponds to the unit of $R$ through the following identification  \[1_R \in R \cong R_a \cong R_a \boxtimes (s^{-1}\overline{\mathcal{C}})^{\boxtimes 0} \boxtimes R_a \subseteq \mathsf{Cobar}^{\boxtimes}(\mathcal{C})(a,a)_0.\] 
Each differential 
\[ D^{\boxtimes} \colon \mathsf{Cobar}^{\boxtimes}(\mathcal{C})(a,b) \to \mathsf{Cobar}^{\boxtimes}(\mathcal{C})(a,b)\]
is induced by the sum $\partial + \Delta + h$. The fact $D^{\boxtimes} \circ D^{\boxtimes}  =0$ then follows from the compatibility of $\partial$ and $\Delta$, the coassociativity of $\Delta$, and the curvature equation relating $h$, $\partial$, and $\Delta$.
The composition of morphisms is given by concatenation of monomials. This construction gives rise to a functor
\[ \mathsf{Cobar}^{\boxtimes} \colon \mathsf{cCoalg}_R \to \mathsf{dgCat}_R.\]

The normalized singular chains on a topological space may be regarded as a categorical coalgebra as we now explain. Recall that the classical normalized singular chains functor 
\[C_{\bullet}: \mathsf{Top} \rightarrow \mathsf{dgCoalg}_R\]
assigns to a space $X$ a dg coalgebra $(C_{\bullet}(X),\partial , \Delta)$ with coproduct $\Delta$ given by the Alexander-Whitney diagonal approximation. This does not define a categorical coalgebra with curvature $0$ since, in general, $\epsilon \circ \partial  \neq 0$, where $\epsilon \colon C_{\bullet}(X) \to C_{0}(X)$ is the canonical projection map. However, one may proceed as follows. Let 
\begin{equation}\label{equation: e C(X) to R}
    e\colon C_\bullet(X)\rightarrow R 
\end{equation}
be the $1$-cochain induced by sending non-degenerate singular $1$-simplices to $1_R$ and everything else to $0$. Define a degree $-1$  linear map\[\widetilde{\partial} \colon C_\bullet(X) \to C_{\bullet-1}(X)\] and a degree $-2$ linear map \[h: C_\bullet(X) \rightarrow R\] by
\[ \widetilde{\partial} = \partial  - (\text{id} \otimes e - e\otimes \text{id}) \circ \Delta \]
and
\[ h = (e \otimes e) \circ \Delta + e \circ \partial, \]
respectively. 
A straightforward check yields that $\mathcal{C}(X)=(C_\bullet(X),\widetilde{\partial},\Delta,h)$ defines a categorical coalgebra. Furthermore, this construction defines a functor 
\[ \mathcal{C} \colon \mathsf{Top} \to \mathsf{cCoalg}_R.\]

\subsubsection{The totalized cosimplicial version \ref{item: version of cobar - many prod}}
For any dg $R$-coalgebra $C$ such that $C_0=R[S]$, the dg category $\mathsf{Cobar}^{\prod}(C)$ is defined as follows. The set of objects is $S$, and for any $a,b\in S$,
\[\mathsf{Cobar}^{\prod}(C)(a,b)=\prod_{n=0}^{\infty} R_a\otimes (s^{-1}C)^{\otimes n}\otimes R_b.\]
Each differential
\[
D^{\prod}\colon  \mathsf{Cobar}^{\prod}(C)(a,b)\to\mathsf{Cobar}^{\prod}(C)(a,b)
\]
is induced by the sum $\partial+\Delta$. The composition of morphisms is given by concatenation of monomials. The chain complex $\mathsf{Cobar}^{\prod}(C)(a,b)$ may also be described as the totalization of a cosimplicial chain complex as follows. Recall that for a cosimplicial chain complex $[n]\mapsto (C(n),\partial)$, its totalization has underlying $R$-module
\[
\mathrm{Tot}^{\prod}_{\bullet}(\{C(n)\}_{n\ge0}) = \prod_{n=0}^{\infty} C(n)_{\bullet + n}
\]
and differential 
\[
D^{\prod} = \partial + \delta,
\] where $\delta$ is a signed sum of the coface maps; see, e.g., \cite[(2.4)]{Wang2024}. Now, consider 
\[
C(n)=C(n;a,b)=R_a\otimes C^{\otimes n}\otimes R_b
\]
with coface maps 
\begin{align*}
    \delta_i\colon R_a\otimes C^{\otimes n-1}\otimes R_b & \to R_a\otimes C^{\otimes n}\otimes R_b\quad (0\leq i\leq n)\\
    c_0\otimes c_1\otimes\dots\otimes c_{n-1}\otimes c_{n} &\mapsto c_0\otimes \dots\otimes c_{i-1}\otimes\Delta (c_i)\otimes c_{i+1}\otimes\dots\otimes c_{n}
\end{align*}
induced by the coproduct $\Delta:C\to C\otimes C$ when $i=1,\ldots, n-1$,
and 
\begin{align*}
\delta_0 \colon c_0\otimes c_1\otimes\dots\otimes c_{n-1}\otimes c_{n}  &\mapsto \rho_{r,a} (c_0) \otimes c_1\otimes\dots\otimes c_{n-1}\otimes c_{n}, \\
\delta_{n} \colon c_0\otimes c_1\otimes\dots\otimes c_{n-1}\otimes c_{n} &\mapsto c_0\otimes c_1\otimes\dots\otimes c_{n-1}\otimes \rho_{l,b}(c_{n}),
\end{align*}
where the map $\rho_{r,a}$ is defined as the composition \[ R_a\cong R \xrightarrow{\cong} R \otimes R \xrightarrow{\text{id} \otimes {i_a}} R \otimes C \cong R_a \otimes C \]
for $i_a \colon R \to C$ being induced by the inclusion $\{a \} \hookrightarrow S$ and $\rho_{l,b} \colon R_b \to  C \otimes R_b$ is defined similarly. The codegeneracy maps are defined by
\begin{align*}
    \sigma_i\colon R_a\otimes C^{\otimes n+1}\otimes R_b & \to R_a\otimes C^{\otimes n}\otimes R_b\quad (0\leq i\leq n)\\
    c_0\otimes c_1\otimes\dots\otimes c_{n+1}\otimes c_{n+2} &\mapsto c_0\otimes \dots\otimes c_{i}\otimes\varepsilon (c_{i+1})\otimes c_{i+2}\otimes\dots\otimes c_{n+2}
\end{align*}
where $\varepsilon:C\twoheadrightarrow C_0\to R$ is the counit map. Then $\mathsf{Cobar}^{\prod}(C)(a,b)$ is defined as the totalization $\mathrm{Tot}^{\prod}\left(\{C(n;a,b)\}_{n\ge0}\right)$.

The \textit{conormalization} of the totalization $\mathrm{Tot}^{\prod}\left(\{C(n;a,b)\}_{n\ge0}\right)$, which we denote by $N^c\mathrm{Tot}^{\prod}\left(\{C(n;a,b)\}_{n\ge0}\right)$, is the subcomplex where any $\sigma_i$ vanishes. The inclusion $N^c\mathrm{Tot}^{\prod}\hookrightarrow\mathrm{Tot}^{\prod}$ is a quasi-isomorphism (cf.~\cite[Lemma 2.5]{Irie}).

\begin{example}
    Let $X$ be a topological space and let $b \in X$. We are interested in the following examples of differential graded coalgebras $C$ equipped with a set of objects $S$ such that $C_0 = R[S]$:
    \begin{itemize}
        \item $C = C_{\bullet}(X)$, with $S$ the set of points in $X$.
        \item $C = C^0_{\bullet}(X, b)$ or $C^1_{\bullet}(X, b)$, with $S = \{b\}$.
    \end{itemize}
    In case $X=M$ is a smooth manifold, we may also restrict the discussion to piecewise smooth singular chains $C^{\mathrm{s}}_{\bullet}$.
\end{example}

\subsection{The natural transformations \texorpdfstring{$\mathcal{A},\mathcal{G}$}{A, G} and the natural map $It_X$.} 
For any topological space $X$, the functors $\mathcal{A}_X$, $\mathcal{G}_X$ act as the identity on objects. In the following, we describe the action of $\mathcal{A}_X$, $\mathcal{G}_X$ and $It_X$ on the morphism complexes.

\subsubsection{The ``many object'' version of Adams' map.}

Denote by $v_0, \ldots, v_n$ the vertices of the standard $n$-simplex $\mathbb{\Delta}^n \subset \mathbb{R}^{n+1}$. Using the method of acyclic models, one may construct a collection of singular cubical chains
\[ \{ \theta_n \colon I^{n-1} \to \mathsf{P}(\mathbb{\Delta}^n)(v_0, v_n) \}_{n\geq 1} \]
such that
\begin{enumerate}
    \item 
$\theta_1(0) \in \mathsf{P}(\mathbb{\Delta}^1)(v_0, v_1)$ is the (Moore) path $\theta_1(0) \colon [0, \sqrt{2}]\to \mathbb{\Delta}^1$ given by 
\[\theta_1(0)(s)= v_0 + \frac{s}{\sqrt{2}}(v_1-v_0),\]
and
\item
\[ \partial^{\square} \circ \theta_n = \sum_{i=1}^{n-1} (-1)^i (\mathsf{P}(l_{n-i,n}) \circ \theta_{n-i})* (\mathsf{P}(f_{i,n}) \circ \theta_i)- \sum_{i=1}^{n-1}(-1)^i\mathsf{P}(d_i) \circ \theta_i,\]
where $l_{j,n} \colon \mathbb{\Delta}^j \hookrightarrow \mathbb{\Delta}^n$ and $f_{j,n} \colon \mathbb{\Delta}^j \hookrightarrow \mathbb{\Delta}^n$ denote the last and first $j$-dimensional face inclusions, respectively,  $d_i \colon \mathbb{\Delta}^{n-1} \hookrightarrow \mathbb{\Delta}^n$ denotes the $i$-th face inclusion, and $*$ denotes concatenation of paths.
\end{enumerate}
See \cite[Section 3.4]{Rivera-Medina} for a description of an explicit choice of maps $\{ \theta_n \}_{n \geq 1}$. This construction goes back to \cite[Section 3]{Adams}. 

Any such collection of maps $\{ \theta_n \}_{n \geq 1}$ gives rise to a well defined linear map of degree $-1$ 
\[A_X \colon \overline{\mathcal{C}}(X) \to \bigoplus_{a,b \in X}\mathsf{C}^{\square}(\mathsf{P}(X))(a,b)\]
by sending a class represented by a singular chain $\sigma \colon \mathbb{\Delta}^n \to X$ to \[A_X(\sigma)=\mathsf{P}(\sigma) \circ \theta_n \colon I^{n-1} \to \mathsf{P}(X)(\sigma(v_0), \sigma(v_n)).\] The map $A_X$ satisfies the (curved) Maurer-Cartan equation
\[ \partial^{\square} \circ A_X - A_X \circ \widetilde{\partial}= * \circ (A_X \otimes A_X) \circ \widetilde{\Delta} + \widetilde{h}, \]
where $\widetilde{h} \colon \overline{\mathcal{C}}(X) \to \bigoplus_{a,b \in X}\mathsf{C}^{\square}(\mathsf{P}(X))(a,b)$ is the degree $-2$ map given by the composition 
\begin{align*}
    \mathcal{C}(X) \xrightarrow{\rho_r}  \mathcal{C}(X) \otimes \mathcal{C}_0(X) \xrightarrow{h \otimes \text{id}} R \otimes  \mathcal{C}_0(X) \cong  \mathcal{C}_0(X)& \\\hookrightarrow  \bigoplus_{a \in X}\mathsf{C}^{\square}(\mathsf{P}(X))(a,a) &\hookrightarrow \bigoplus_{a,b \in X}\mathsf{C}^{\square}(\mathsf{P}(X))(a,b).
\end{align*} 
Extending $A_X$ to be compatible with composition, for any $a,b \in X$, we obtain a natural degree $0$ linear map 
\[\mathcal{A}_X \colon \mathsf{Cobar}^{\boxtimes}(\mathcal{C}(X))(a,b) \to \mathsf{C}^{\square}(\mathsf{P}(X))(a,b),\]
which, by the Maurer-Cartan equation above, is a chain map.

\subsubsection{A formal dual of Chen's iterated integral map} \label{Section: A formal dual of Chen's iterated integral map}
Given $a,b\in X$, the chain map 
    \begin{equation}\label{equation: It_X}
        It_X \colon \mathsf{C}^{\square}(\mathsf{P}X)(a,b) \to \mathsf{Cobar}^{\prod}(C(X))(a,b)
    \end{equation}
is defined as the composition of several maps specified below.

\emph{Step 1.} There is a natural chain map
\[
\eta^{\square,\mathbb{\Delta}} \colon \mathsf{C}^{\square}_{\bullet}(\mathsf{P}X)(a,b)
\to \mathsf{C}^{\mathbb{\Delta}}_{\bullet}(\mathsf{P}X)(a,b)
= C_{\bullet}(\mathsf{P}X(a,b)),
\]
induced by the standard triangulation of cubes. Concretely,
for any singular cube $\lambda \colon [0,1]^n \to \mathsf{P}X(a,b)$,
\[
\eta^{\square,\mathbb{\Delta}}(\lambda)
= \sum_{\tau \in S_n} (-1)^{\mathrm{sgn}(\tau)}\, \lambda \circ \iota_{\tau},
\]
where $S_n$ denotes the symmetric group on $n$ letters, and
\[
\iota_\tau \colon \mathbb{\Delta}^n \to [0,1]^n,\qquad
(t_1,\dots,t_n) \mapsto (t_{\tau(1)}, \dots, t_{\tau(n)})
\]
for $0 \le t_1 \le \cdots \le t_n \le 1$. More generally,
\[
\eta^{\square,\mathbb{\Delta}} \colon C^{\square}_{\bullet} \Rightarrow C^{\mathbb{\Delta}}_{\bullet}
\]
is a natural transformation from the normalized cubical singular chain functor on topological spaces to the normalized simplicial singular chain functor; the map displayed above is its component at the space $\mathsf{P}X(a,b)$ and these are natural with respect to continuous maps of spaces.

\emph{Step 2.} Consider the cosimplicial space 
\[
[n] \longmapsto \mathsf{P}X(a,b) \times \mathbb{\Delta}^n
\]
with cosimplicial structure induced from the standard one
$[n] \mapsto \mathbb{\Delta}^n$.  
Set 
\[
u_n \coloneq \mathrm{id}_{\mathbb{\Delta}^n} \in C_n(\mathbb{\Delta}^n),\quad u \coloneq \{u_n\}_{n \ge 0}.
\] 
Then, the linear maps
\[
C_{\bullet}(\mathsf{P}X(a,b)) \to C_{\bullet+n}(\mathsf{P}X(a,b) \times \mathbb{\Delta}^n),
\quad x \mapsto (-1)^nx \times u_n,
\]
for all $n\ge0$ together induce a chain map
\[
\Phi_{u} \colon C_{\bullet}(\mathsf{P}X(a,b)) \to 
\mathrm{Tot}^{\prod}_{\bullet}(\{C(\mathsf{P}X(a,b) \times \mathbb{\Delta}^n)\}_{n\ge0}).
\]
More generally, any choice of $u = \{u_n \in C_n(\mathbb{\Delta}^n)\}_{n \ge 0}$ that satisfies
\begin{equation}\label{equation: condition for u_0}
[u_0] = [\mathrm{id}_{\mathbb{\Delta}^0}] \in H_0(\mathbb{\Delta}^0) \quad\text{and}\quad
\partial u_n = \sum_{i=0}^n (-1)^i (d_i)_*(u_{n-1}) \quad (\forall n \ge 1)
\end{equation}
suffices to define $\Phi_u$. Moreover, by an acyclic models argument, using the fact that $H_{n+1}(\mathbb{\Delta}^n)=0$ for all $n \ge 0$, 
one can show that the chain homotopy class of $\Phi_u$ is independent of the choice of $u$.

\emph{Step 3.} Consider the cosimplicial space $n\mapsto \{a\}\times X^{n}\times\{b\}$ with cofaces induced by the diagonal map and codegeneracies given by forgetful maps.
The evaluation maps
\begin{align*}
\mathrm{Ev}_n\colon \quad\mathsf{P}X(a,b)\times\mathbb{\Delta}^n &\to \{a\}\times X^{n}\times\{b\}, \\
((\gamma,T),(t_1,\dots,t_n)) &\mapsto (a,\gamma(t_1T),\dots,\gamma(t_nT),b)
\end{align*}
for all $n\geq0$ respect cosimplicial structures, inducing a chain map
\[
\mathrm{Ev}_* \colon \mathrm{Tot}^{\prod}(\{C(\mathsf{P}X(a,b) \times \mathbb{\Delta}^n)\}_{n\ge0})\to\mathrm{Tot}^{\prod}(\{C(\{a\}\times X^{n}\times\{b\})\}_{n\ge0}).
\]

\emph{Step 4.} Iterations of the standard Alexander-Whitney map define a cosimplicial chain map
\[
AW_n \colon C_{\bullet}(\{a\}\times X^n\times\{b\})\to (R_a\otimes C(X)^{\otimes n}\otimes R_b)_{\bullet},\quad n\geq0,
\]
inducing a chain map
\[
AW \colon \mathrm{Tot}^{\prod}_{\bullet}(\{C(\{a\}\times X^{n}\times\{b\})\}_{n\ge0}) \to \mathsf{Cobar}^{\prod}(C(X))(a,b).
\]

We have thus defined \eqref{equation: It_X} as the composition 
\[
It_X = AW\circ\mathrm{Ev}_*\circ \Phi_{u}\circ\eta^{\square,\mathbb{\Delta}}.
\]

\begin{remark}\label{remark: when to triangulate cubes}
In the definition of $It_X$, we chose to triangulate cubes right from the beginning.  
Alternatively, one may remain in the setting of normalized cubical singular chains and pass to normalized simplicial singular chains later, either before $\mathrm{Ev}_*$ or before $AW$. In this case,
\[
It_X = AW \circ \eta^{\square,\mathbb{\Delta}} \circ \mathrm{Ev}_* \circ \Phi_{v} = AW \circ  \mathrm{Ev}_* \circ \eta^{\square,\mathbb{\Delta}} \circ\Phi_{v},
\]
where $v = \{\,v_n \in C^{\square}_n(\mathbb{\Delta}^n)\,\}_{n \ge 0}$ is chosen such that $v_0$ is the unique point map and
\[
\partial^{\square} v_n = \sum_{i=0}^n (-1)^i (d_i)_*(v_{n-1}) \quad \text{for all } n \ge 1.
\]
Lemma~\ref{lemma: simplex to cube} below, together with the fact that $\eta^{\square,\mathbb{\Delta}}$ intertwines the cross products on cubes and simplices, implies that these two approaches are equivalent:
\[
\mathrm{Ev}_* \circ \Phi_{u} \circ \eta^{\square,\mathbb{\Delta}} = \eta^{\square,\mathbb{\Delta}} \circ \mathrm{Ev}_* \circ \Phi_{v} =  \mathrm{Ev}_* \circ \eta^{\square,\mathbb{\Delta}} \circ \Phi_{v}
\]
if $v = \eta^{\mathbb{\Delta},\square}(u)$ or $u = \eta^{\square,\mathbb{\Delta}}(v)$. Note that $v = \eta^{\mathbb{\Delta},\square}(u)$ acturally implies $u = \eta^{\square,\mathbb{\Delta}}(v)$ by Lemma~\ref{lemma: simplex to cube}.
\end{remark}

\begin{lemma}\label{lemma: simplex to cube}
    There is a natural transformation 
    \[
    \eta^{\mathbb{\Delta},\square}\colon C^{\mathbb{\Delta}}_{\bullet}\Rightarrow C^{\square}_{\bullet}
    \] 
    from the normalized simplicial singular chain functor to the normalized cubical singular chain functor, such that     
    \[\eta^{\square,\mathbb{\Delta}}\circ\eta^{\mathbb{\Delta},\square} = \mathrm{id}_{C^{\mathbb{\Delta}}_{\bullet}}.
    \]
\end{lemma}
\begin{proof}
For every $n \ge 0$, define a folding map
\[
f_n \colon I^n \to \Delta^n, \qquad
f_n(t_1, \dots, t_n) = (t_1', \dots, t_n'), \quad
t_k' = \max\{t_1, \dots, t_k\}.
\]
A similar map appears in \cite[Section~4.2]{igusa2009}, although for a different purpose.

For any topological space $Y$ and any singular simplex $\sigma \colon \Delta^n \to Y$, define
\[
\eta^{\mathbb{\Delta},\square}_Y(\sigma) = \sigma \circ f_n \colon I^n \to Y,
\]
and extend linearly to $C^{\mathbb{\Delta}}_\bullet(Y)$. By construction, $\eta^{\mathbb{\Delta},\square}_Y$ is natural in $Y$.

To see that $\eta^{\mathbb{\Delta},\square}_Y$ is a chain map, observe that among the cubical faces of $\eta^{\mathbb{\Delta},\square}_Y(\sigma)$, those of the form $\{t_k=1\}$ with $k<n$ are degenerate, while the remaining $n+1$ non-degenerate faces correspond exactly to the simplicial faces of $\sigma$.

Finally, to verify that $\eta^{\square,\mathbb{\Delta}}_Y \circ \eta^{\mathbb{\Delta},\square}_Y = \mathrm{id}_{C^{\mathbb{\Delta}}_\bullet(Y)}$, compute
\[
(\eta^{\square,\mathbb{\Delta}}_Y \circ \eta^{\mathbb{\Delta},\square}_Y)(\sigma) 
= \sum_{\tau \in S_n} \mathrm{sgn}(\tau)\, \sigma \circ f_n \circ \iota_\tau
= \sigma,
\]
since $f_n \circ \iota_\tau$ is degenerate whenever $\tau \ne \mathrm{id}$, and $f_n \circ \iota_{\mathrm{id}} = \mathrm{id}_{\Delta^n}$.
\end{proof}

\begin{remark}\label{remark: I formal dual to Chen}
Let $M$ be a smooth manifold, $\Omega^{\bullet}(M)$ the dg algebra of differential forms on $M$, and $\mathsf{P}^{\mathrm{s}}M(a,b)$ the space of piecewise smooth Moore loops in $M$ from $a$ to $b$. Set $R=\mathbb{R}$. The iterated integral map for $\mathsf{P}^{\mathrm{s}}M(a,b)$, originally due to Chen in \cite{Chen}, is a cochain map              
\begin{equation}\label{equation: Chen int}
    \int \colon \mathsf{Bar}(\mathbb{R}_a,\Omega(M),\mathbb{R}_b) \to \Omega^{\bullet}(\mathsf{P}^{\mathrm{s}}M(a,b)),
\end{equation}
where $\mathsf{Bar}$ denotes a version of the classical bar construction and differential forms on $\mathsf{P}^{\mathrm{s}}M(a,b)$ are defined via the framework of of \textit{differentiable spaces}. Following the presentation in \cite{GETZLER1991339}, $\int$ is induced by a sequence of maps
\begin{align*}
    \mathbb{R}_a \otimes \Omega^{\bullet}(M)^{\otimes n} \otimes \mathbb{R}_b & \hookrightarrow \Omega^{\bullet}(\{a\}\times M^n \times \{b\}) \\ 
    &\xrightarrow{\mathrm{Ev}_n^*} \Omega^{\bullet}(\mathsf{P}^{\mathrm{s}}M(a,b)\times \mathbb{\Delta}^n) \xrightarrow{\int_{\mathbb{\Delta}^n}} \Omega^{\bullet - n}(\mathsf{P}^{\mathrm{s}}M(a,b))
\end{align*}
for $n\geq0$, where $\int_{\mathbb{\Delta}^n}$ denotes integration along the fibers. Consider the pairing
\[
\Omega^{\bullet}(Y)\times C^{\mathrm{s}}_{\bullet}(Y) \to \mathbb{R}
\]
induced by integration, where $Y$ is $M$ or $\mathsf{P}^{\mathrm{s}}M(a,b)$, and the pairing
\begin{align}\label{equation: pairing bar cobar}
&\mathsf{Bar}(\mathbb{R}_a,\Omega^{\bullet}(M),\mathbb{R}_b) \times \mathsf{Cobar}^{\prod}(C^{\mathrm{s}}_{\bullet}(M))(a,b) \to \mathbb{R} \\ \nonumber
    &\Big\langle \sum_{n=0}^{N}1\otimes\omega_{n,1}\otimes\dots\otimes\omega_{n,n}\otimes1\,,\,\big(1\otimes\alpha_{m,1}\otimes\dots\otimes\alpha_{m,m}\otimes 1\big)_{m\ge0} \Big\rangle \\ \nonumber
    &=  \sum_{n=1}^N \langle\omega_{n,1},\alpha_{n,1}\rangle\cdots\langle\omega_{n,n},\alpha_{n,n}\rangle
\end{align}
induced by summing up integrations levelwise.
There is a chain map 
\[
It_M\colon C^{\mathrm{s}}_{\bullet}(\mathsf{P}^{\mathrm{s}}M(a,b)) \to \mathsf{Cobar}^{\prod}_{\bullet}(C^{\mathrm{s}}(M))(a,b)
\]
defined analogously to \eqref{equation: It_X}.
The maps $It_M$ and $\int$ \eqref{equation: Chen int} are formally dual in the sense that
\[
\Big\langle \int\boldsymbol{\omega}\, , \,\alpha \Big\rangle = \Big\langle \boldsymbol{\omega}
\, , \,It_M(\alpha) \Big\rangle
\]
for any $\boldsymbol{\omega}\in\mathsf{Bar}(\mathbb{R}_a,\Omega^{\bullet}(M),\mathbb{R}_b)$ and $\alpha\in C^{\mathrm{s}}_{\bullet}(\mathsf{P}^{\mathrm{s}}M(a,b))$.
\end{remark}

Now set $a = b$, and denote the product on $C^{\mathrm{s}}_{\bullet}(\mathsf{P}^{\mathrm{s}}M(a,a))$ by $\times$. For brevity, write $1 \otimes \omega_1 \otimes \dots \otimes \omega_n \otimes 1$ as $\omega_1\cdots\omega_n$. For any $\alpha, \beta \in C^{\mathrm{s}}_{\bullet}(\mathsf{P}^{\mathrm{s}}M(a,a))$ and $\omega_1, \dots, \omega_n \in \Omega(M)$, we have
\begin{align}\label{equation: iterated integration dga map pairing}
    \Big\langle \omega_1\cdots\omega_n\, , \, It_M(\alpha\times\beta) \Big\rangle & = \Big\langle \int \omega_1\cdots\omega_n \, , \, \alpha\times\beta \Big\rangle \\ \nonumber
    & = \sum_{i=1}^n \Big\langle \int \omega_1\cdots \omega_i\,,\,\alpha\Big\rangle \Big\langle \int \omega_{i+1}\cdots \omega_n\,,\,\beta \Big\rangle \\ \nonumber
    & = \sum_{i=1}^n \Big\langle \omega_1\cdots \omega_i\,,\,It_M(\alpha)\Big\rangle \Big\langle \omega_{i+1}\cdots \omega_n\,,\,It_M(\beta) \Big\rangle \\ \nonumber
    & = \Big\langle \omega_1\cdots \omega_n\, , \,It_M(\alpha)It_M(\beta) \Big\rangle,
\end{align}
where the equality on the second line is exactly \cite[(1.6.2)]{Chen}, and the last line follows from the fact that the concatenation product on $\mathsf{Cobar}^{\prod}(C^{\mathrm{s}}_{\bullet}(M))(a,b)$ is dual to the deconcatenation coproduct on $\mathsf{Bar}(\mathbb{R}_a,\Omega^{\bullet}(M),\mathbb{R}_a)$ under the pairing~\eqref{equation: pairing bar cobar}.

Equation~\eqref{equation: iterated integration dga map pairing} suggests that $It_M$ preserves the products (compositions) up to terms that vanish under the integration pairing with differential forms. A more general and precise statement is Theorem~\ref{theorem: A infinity map I}, which we prove in Section~\ref{Section: proof of theorem A infinity extension}.

\subsubsection{The natural transformation $\mathcal{G}$}

For any $a,b\in X$, define a degree 0 linear map
\[
    G_X \colon R_a\boxtimes s^{-1}\overline{\mathcal{C}}(X)\boxtimes R_b \to \mathsf{Cobar}^{\prod}(C(X))(a,b)
\]
such that for any $\sigma\colon \Delta^n\to X$,
\[
    G_X(s^{-1}\sigma) = s^{-1}\sigma + e(\sigma) \in (R_a\otimes s^{-1}C(X)\otimes R_b)\oplus (R_a\otimes R\otimes R_b),
\]
where $e$ is the map~\eqref{equation: e C(X) to R}. Extending $G_X$ to be compatible with compositions, and using the identification
\[
R_a\boxtimes (s^{-1}\overline{\mathcal{C}}(X))^{\boxtimes 0}\boxtimes R_b \cong R \cong R_a\otimes R\otimes R_b = R_a\otimes (s^{-1}C(X))^{\otimes 0}\otimes R_b,
\]
we obtain a natural degree 0 linear map
\begin{equation}\label{equation: G_X}
    \mathcal{G}_X\colon \mathsf{Cobar}^{\boxtimes}(\mathcal{C}(X))(a,b)\to \mathsf{Cobar}^{\prod}(C(X))(a,b)
\end{equation}
for any $a,b\in X$, which is compatible with compositions. A straightforward calculation shows that $\mathcal{G}_X$ is a chain map.

\section{Construction of $\mathcal{I}$ (Theorem~\ref{theorem: A infinity map I}) and $\mathcal{H}$ (Theorem~\ref{theorem: A infinity homotopy})} 
\label{Section: acyclic model arguments}

This section establishes the existence of a natural $A_\infty$-transformation $\mathcal{I}$ extending $\{It_X\}_{X\in\mathsf{Top}}$, and of an $A_\infty$-homotopy $\mathcal{H}$ between $\mathcal{F}=\mathcal{I}\circ\mathcal{A}$ and $\mathcal{G}$. 
Both $\mathcal{I}$ and $\mathcal{H}$ are constructed via the method of acyclic models; the construction of $\mathcal{I}$ is more geometric in nature, whereas the construction of $\mathcal{H}$ is more algebraic. 
At the end of this section we provide a precise formulation of the statement that $\mathcal{I}$ is a left $A_\infty$-homotopy inverse of $\mathcal{A}$ for simply connected spaces, as announced in Remark~\ref{remark: Chen Adams inverse simply connected}.

We begin by fixing $A_\infty$ sign conventions. Our sign conventions for $A_\infty$-categories and $A_\infty$-functors follow those for $A_\infty$-algebras and $A_\infty$-morphisms in \cite{loday-vallette2012algebraic} and differ from those in \cite{lefevre2002categories}. The difference essentially arises from reversing the order of inserting objects in the multilinear maps, which, together with the Koszul sign convention for reordering graded objects, leads to a straightforward conversion rule between the two conventions. Applying this rule to the $A_\infty$-algebra and $A_\infty$-morphism signs in \cite{lefevre2002categories} recovers exactly the conventions of \cite{loday-vallette2012algebraic}. Since \cite{loday-vallette2012algebraic} does not specify signs for $A_\infty$-homotopies, we extend this conversion rule to the $A_\infty$-homotopy signs from \cite{lefevre2002categories}, ensuring all signs remain compatible with \cite{loday-vallette2012algebraic}.

For an $A_\infty$-category $\mathsf{C}$ and objects $X_0,\dots,X_n \in \mathrm{Ob}(\mathsf{C})$, the structure maps
\[
m_{n}^{\mathsf{C}} \colon \mathrm{Hom}^{\mathsf{C}}(X_0,X_1) \otimes \cdots \otimes \mathrm{Hom}^{\mathsf{C}}(X_{n-1},X_n) \to \mathrm{Hom}^{\mathsf{C}}(X_0,X_n)
\]
are of degree $n-2$, satisfying the $A_\infty$-relations
\[
\sum_{p+q+r=n} (-1)^{p+qr} 
m_{p+1+r}^{\mathsf{C}} \circ \big( 1^{\otimes p} \otimes m_q^{\mathsf{C}} \otimes 1^{\otimes r} \big) = 0, \quad \forall n \geq 1.
\]

For two $A_\infty$-categories $\mathsf{C}, \mathsf{C}'$, an $A_\infty$-functor $\mathsf{F} \colon \mathsf{C} \to \mathsf{C}'$ consists of a map 
\[
F \colon \mathrm{Ob}(\mathsf{C}) \to \mathrm{Ob}(\mathsf{C}')
\]
on objects and a collection of degree $n-1$ linear maps
\[
F_n \colon \mathrm{Hom}^{\mathsf{C}}(X_0,X_1) \otimes \cdots \otimes \mathrm{Hom}^{\mathsf{C}}(X_{n-1},X_n) \to \mathrm{Hom}^{\mathsf{C}'}\big(F(X_0),F(X_n)\big)
\]
for all $n \geq 1$ and $X_0,\dots,X_n \in \mathrm{Ob}(\mathsf{C})$, satisfying that for any $n \geq 1$,
\[
\sum_{p+q+r=n} (-1)^{p+qr} 
F_{p+1+r} \big( 1^{\otimes p} \otimes m_q^{\mathsf{C}} \otimes 1^{\otimes r} \big)
=
\sum_{\substack{k \geq 1 \\ i_1+\cdots+i_k = n}} (-1)^{\epsilon} 
m_k^{\mathsf{C}'} \big( F_{i_1} \otimes \cdots \otimes F_{i_k} \big),
\]
where
\[
\epsilon = \sum_{j=1}^k (k-j)(i_j - 1).
\]

For two $A_\infty$-functors $\mathsf{F}, \mathsf{F}' \colon \mathsf{C} \to \mathsf{C}'$ associated with the same object map $F$, an $A_\infty$-homotopy $\mathsf{H} \colon \mathsf{F} \Rightarrow \mathsf{F}'$ consists of a collection of degree $n$ linear maps
\[
H_n \colon \mathrm{Hom}^{\mathsf{C}}(X_0,X_1) \otimes \cdots \otimes \mathrm{Hom}^{\mathsf{C}}(X_{n-1},X_n) \to \mathrm{Hom}^{\mathsf{C}'}\big(F(X_0),F(X_n)\big)
\]
for all $n \geq 1$ and $X_0,\dots,X_n \in \mathrm{Ob}(\mathsf{C})$, satisfying that for any $n \geq 1$,
\begin{align*}
F_n - F'_n &= \sum_{p+q+r=n} (-1)^{p+qr} 
H_{p+1+r} \big( 1^{\otimes p} \otimes m_q^{\mathsf{C}} \otimes 1^{\otimes r} \big) \\
&\quad + \sum_{\substack{k \geq l \geq 1 \\ i_1+\cdots+i_k = n}} (-1)^{\delta} 
m_k^{\mathsf{C}'} \Big(
G_{i_1} \otimes \cdots \otimes G_{i_{l-1}} \otimes H_{i_l} \otimes F_{i_{l+1}} \otimes \cdots \otimes F_{i_k}
\Big),
\end{align*}
where
\[
\delta = k-1 + \sum_{j=1}^{k} (k-j)(i_j - 1).
\]

Under the above conventions, we focus on dg categories, i.e., $A_\infty$-categories where $m_1$ is the differential, $m_2$ the composition, and $m_n=0$ for all $n>2$.

\subsection{Proof of Theorem \ref{theorem: A infinity map I}} \label{Section: proof of theorem A infinity extension} 
Recall from Section \ref{Section: A formal dual of Chen's iterated integral map} that $It_X$ is the composition of natural chain maps
\begin{align}\nonumber
    \mathsf{C}^{\square}_{\bullet}(\mathsf{P}X)(a,b)    &\xrightarrow{\Phi_{v}} \mathrm{Tot}^{\prod}_{\bullet}(\{C^{\square}(\mathsf{P}X(a,b) \times \mathbb{\Delta}^n)\}_{n\ge0}) \\ \nonumber
     & \xrightarrow{\eta^{\square,\mathbb{\Delta}}} \mathrm{Tot}^{\prod}_{\bullet}(\{C(\mathsf{P}X(a,b) \times \mathbb{\Delta}^n)\}_{n\ge0}) \\ \label{equation: It_X original composition definition} & \xrightarrow{\mathrm{Ev}_*}  \mathrm{Tot}^{\prod}_{\bullet}(\{C(\{a\}\times X^{n}\times\{b\})\}_{n\ge0})  \xrightarrow{AW} \mathsf{Cobar}^{\prod}_{\bullet}(C(X))(a,b).
\end{align}

Clearly, $AW$ defines a functor between dg categories, natural in $X \in \mathsf{Top}$, where the composition on
\[
\big\{\mathrm{Tot}^{\prod}_{\bullet}(\{C(\{a\}\times X^{n}\times\{b\})\}_{n\ge0})\big\}_{a,b\in X}
\]
is induced by the Cartesian product of spaces:
\[
\{a\}\times X^{n_1}\times\{b\}\times\{b\}\times X^{n_2}\times\{c\} \to \{a\}\times X^{n_1+n_2}\times\{c\}.
\]
Hence it suffices to show that $\mathrm{Ev}_*\circ\eta^{\square,\mathbb{\Delta}}\circ\Phi_{v}$ is an $A_\infty$-functor between dg categories. This naturally leads one to seek associative compositions on the families
\[
\big\{\mathrm{Tot}^{\prod}_{\bullet}(\{C^{\square}(\mathsf{P}X(a,b)\times \mathbb{\Delta}^n)\}_{n\ge0})\big\}_{a,b\in X}
, \quad
\big\{\mathrm{Tot}^{\prod}_{\bullet}(\{C(\mathsf{P}X(a,b)\times \mathbb{\Delta}^n)\}_{n\ge0})\big\}_{a,b\in X},
\]
respectively, and to analyze $\mathrm{Ev}_*$, $\eta^{\square,\mathbb{\Delta}}$, and $\Phi_v$ separately.
The most natural candidate for such compositions would be induced by space maps
\[
\mathsf{P}X(a,b)\times\mathbb{\Delta}^{n_1}\times\mathsf{P}X(b,c)\times\mathbb{\Delta}^{n_2} \to \mathsf{P}X(a,c)\times\mathbb{\Delta}^{n_1+n_2},
\]
defined by concatenating paths and gluing points in $\mathbb{\Delta}^{n_1},\mathbb{\Delta}^{n_2}$ proportionally to their respective path lengths. However, such maps fail to exist when both paths have zero length, so the composition is only partially defined.

To settle this issue, we introduce a cosimplicial space
\begin{align}\label{equation: P^nX(a,b)}
    [n] \mapsto \mathsf{P}^nX(a,b) \coloneqq \big\{(\gamma,T,&t_1,\dots,t_n) \mid \\ \nonumber &  (\gamma,T) \in \mathsf{P}X(a,b),\ 0 \le t_1 \le \dots \le t_n \le T\big\},
\end{align}
whose cosimplicial structure parallels that of both $\mathsf{P}X(a,b)\times\mathbb{\Delta}^n$ and $\{a\} \times X^{n} \times \{b\}$; see also \cite[(4.2a),(4.2b)]{Wang2024}.
The families
\[
\big\{\mathrm{Tot}^{\prod}_{\bullet}(\{C^{\square}(\mathsf{P}^nX(a,b))\}_{n\ge0})\big\}_{a,b\in X},\quad
\big\{\mathrm{Tot}^{\prod}_{\bullet}(\{C(\mathsf{P}^nX(a,b))\}_{n\ge0})\big\}_{a,b\in X}
\]
both admit associative compositions induced by path concatenation together with the obvious gluing of marked points, and therefore form the morphism sets of dg categories.
Recognizing $[n]\mapsto \mathsf{P}^{n}X$ as a cosimplicial analogue of the nerve $\mathsf{N}(\mathsf{P}X)$ of $\mathsf{P}X$, we denote the resulting dg categories by
\[
(\mathsf{Tot}\,C^{\square}\circ\mathsf{cN})(\mathsf{P}X) \quad \text{and} \quad (\mathsf{Tot}\,C^{\mathbb{\Delta}}\circ\mathsf{cN})(\mathsf{P}X).
\]

There is a family of natural maps
\begin{align*}
    q_n \colon \mathsf{P}X(a,b) \times \mathbb{\Delta}^n &\to \mathsf{P}^nX(a,b)\\
    (\gamma,T,t_1,\dots,t_n) &\mapsto (\gamma,T,t_1 T,\dots,t_n T)
\end{align*}
respecting cosimplicial structures, and a family of evaluation maps
\begin{align*}
    \mathrm{ev}_n \colon \mathsf{P}^nX(a,b) &\to \{a\} \times X^n \times \{b\}\\
    (\gamma,T,t_1,\dots,t_n) &\mapsto (a,\gamma(t_1),\dots,\gamma(t_n),b)
\end{align*}
respecting cosimplicial structures and compositions, such that $\mathrm{Ev}_n = \mathrm{ev}_n \circ q_n$.  
Thus, we have a commutative diagram of natural chain maps:
\begin{equation}\label{equation: diagram J_X,1}
\begin{tikzcd}
\mathsf{C}^{\square}_{\bullet}(\mathsf{P}X)(a,b) \arrow[r,"q_*\circ \Phi_{v}"] \arrow[d,"\Phi_{v}"'] &  \mathrm{Tot}^{\prod}_{\bullet}\big(\{C^{\square}(\mathsf{P}^nX(a,b))\}_{n\ge0}\big)\arrow[d,"\eta^{\square,\mathbb{\Delta}}"] \\
\mathrm{Tot}^{\prod}_{\bullet}\big(\{C^{\square}(\mathsf{P}X(a,b) \times \mathbb{\Delta}^n)\}_{n\ge0}\big) \arrow[d,"\eta^{\square,\mathbb{\Delta}}"]              & \mathrm{Tot}^{\prod}_{\bullet}\big(\{C(\mathsf{P}^nX(a,b))\}_{n\ge0}\big)\arrow[d,"\mathrm{ev}_*"]      \\
\mathrm{Tot}^{\prod}_{\bullet}\big(\{C(\mathsf{P}X(a,b) \times \mathbb{\Delta}^n)\}_{n\ge0}\big) \arrow[r,"\mathrm{Ev}_*"]               & \mathrm{Tot}^{\prod}_{\bullet}(\{C(\{a\}\times X^{n}\times\{b\})\}_{n\ge0}),
\end{tikzcd}
\end{equation}
where $\mathrm{ev}_*$ defines a functor between dg categories.

Now, Theorem~\ref{theorem: A infinity map I} reduces to the following lemma.

\begin{lemma} \label{lemma: J_1 extends A infinity}
For any topological space $X$, the natural chain map
\[
\mathcal{J}_{X,1} \coloneqq \eta^{\square,\mathbb{\Delta}}\circ q_* \circ \Phi_v
\]
in \eqref{equation: diagram J_X,1} extends to a natural $A_\infty$-functor
\[
\mathcal{J}_{X} = \{\mathcal{J}_{X,k}\}_{k \ge 1} \colon 
\mathsf{C}^{\square}(\mathsf{P}X)
\to 
(\mathsf{Tot}\,C^{\mathbb{\Delta}} \circ \mathsf{cN})(\mathsf{P}X).
\]  
Explicitly, there exists, for each $k \ge 1$ and $a_0, \dots, a_k \in X$, a natural linear map
\[
\mathcal{J}_{X,k} \colon 
\bigotimes_{i=0}^k 
\mathsf{C}^{\square}(\mathsf{P}X)(a_{i-1}, a_i) 
\to 
\mathrm{Tot}^{\prod}\big(\{ C(\mathsf{P}^n X(a_0,a_k)) \}_{n \ge 0} \big)
\]
of degree $k-1$; and the collection $\{\mathcal{J}_{X,k}\}_{k\ge1}$ satisfies, for any $k \ge 1$,
\begin{align*}
 D^{\prod} \circ \mathcal{J}_{X,k} \;=\; &  \sum_{p+r = k-1} (-1)^{k-1} \mathcal{J}_{X,k} \circ 
(1^{\otimes p} \otimes \partial^{\square} \otimes 1^{\otimes r}) \\ 
  & + \sum_{p+r = k-2} (-1)^{p} \mathcal{J}_{X,k-1} \circ 
(1^{\otimes p} \otimes \mu \otimes 1^{\otimes r}) \\
& + \sum_{i+j = k} (-1)^{i} \mu^{\prod} \circ 
(\mathcal{J}_{X,i} \otimes \mathcal{J}_{X,j}),
\end{align*}
where $\mu$ and $\mu^{\prod}$ denote the respective compositions.
\end{lemma}

Once Lemma~\ref{lemma: J_1 extends A infinity} is proved, Theorem~\ref{theorem: A infinity map I} follows by setting 
\[
\mathcal{I}_{X}\coloneq AW\circ\mathrm{ev}_*\circ\mathcal{J}_X.
\]

In preparation for the proof of Lemma~\ref{lemma: J_1 extends A infinity}, for any $T \ge 0$ and integer $n \ge 0$, set
\[
\mathbb{\Delta}^n_T \coloneq \{ (t_1,\dots,t_n) \in \mathbb{R}^n \mid 0 \le t_1 \le \dots \le t_n \le T \},
\]
which is the standard $n$-simplex scaled by $T$. In particular, $\mathbb{\Delta}^n_0$ consists of a single point 0.
For any $T \ge 0$, the assignment $[n] \mapsto \mathbb{\Delta}^n_T$ forms a cosimplicial space analogous to the standard one at $T=1$.
Indeed, $\{\mathbb{\Delta}^n_T\}_{n \ge 0}$ identifies with the cosimplicial subspace of $\{\mathsf{P}^n\{a\}(a,a)\}_{n \ge 0}$ consisting of constant marked paths of length $T$ in the singleton space $\{a\}$.
The family $\{\mathbb{\Delta}^n_T\}_{n \ge 0,\, T \ge 0}$ inherits an associative composition
\[
\circ : \mathbb{\Delta}^{n_1}_{T_1} \times \mathbb{\Delta}^{n_2}_{T_2} \to \mathbb{\Delta}^{n_1+n_2}_{T_1 + T_2}
\]
from $\{\mathsf{P}^n\{a\}(a,a)\}_{n \ge 0}$, inducing an associative composition 
\[
\circ : \mathrm{Tot}^{\prod}_{\bullet}(\{C^{\square}(\mathbb{\Delta}^n_{T_1})\}_{n \ge 0}) \otimes 
\mathrm{Tot}^{\prod}_{\bullet}(\{C^{\square}(\mathbb{\Delta}^n_{T_2})\}_{n \ge 0}) \to 
\mathrm{Tot}^{\prod}_{\bullet}(\{C^{\square}(\mathbb{\Delta}^n_{T_1 + T_2})\}_{n \ge 0}),
\]
which is a chain map.

The sequence $v=\{v_n\in C^{\square}_n(\mathbb{\Delta}^n)\}_{n\ge0}$ in Remark~\ref{remark: when to triangulate cubes} identifies with a 0-cycle
\[
\tilde{v}=(\tilde{v}_n)_{n\ge0}=((-1)^nv_n)_{n\ge0} \in \mathrm{Tot}^{\prod}_{\bullet}(\{C^{\square}(\mathbb{\Delta}^n)\}_{n \ge 0})
\]
such that $v_0$ is the fudamental cycle of $\mathbb{\Delta}^0$. For any $T\ge0$, denote by  
\[
s_T\colon\mathbb{\Delta}^n\to\mathbb{\Delta}^n_T
\]
the scaling map, and define
\[
  \xi^{0}\langle T\rangle \coloneq (s_T)_*(\tilde{v}) \in \mathrm{Tot}^{\prod}_{0}(\{C^{\square}(\mathbb{\Delta}^n_T)\}_{n \ge 0}).
\]

\begin{lemma}\label{lemma: xi^k}
  There exists a family of chains
  \[
   \left\{\xi^{k}\langle\lambda_0
\mid\dots\mid\lambda_k\rangle \in \mathrm{Tot}^{\prod}_{k}\big(\{C^{\square}(\mathbb{\Delta}^n_{\lambda_0+\dots+\lambda_k})\}_{n\ge0}\big)\right\}_{k\ge0,\,\lambda_0,\dots,\lambda_k\ge0}
  \]
  with $\xi^{0}\langle\lambda_0\rangle$ defined above, satisfying the following two properties:
  \begin{enumerate}
      \item      
      For any $k\ge0$ and $\lambda_0,\dots,\lambda_k\in[0,\infty)$, the chain $\xi^{k}\langle\lambda_0\mid\dots\mid\lambda_k\rangle$ depends continuously on $\lambda_0,\dots,\lambda_k$. More precisely, for any $k,n \ge 0$, there exists an integer $N(k,n)\ge 1$, and for each $1 \le i \le N(k,n)$ a scalar $c^{k}_{n,i} \in R$ and a continuous map 
      \[
      \tau^{k}_{n,i} \colon I^{k+n}\times [0,\infty)^{k+1}\to\mathbb{R}^n,
      \]
      such that: 
      for any $x\in I^{k+n}$ and $\lambda_0,\dots,\lambda_k \in [0,\infty)$,
      $
      \tau^{k}_{n,i}(x,\lambda_0,\dots,\lambda_k)$ lies in $\mathbb{\Delta}^n_{\lambda_0+\cdots+\lambda_k};
      $
      and for any $\lambda_0,\dots,\lambda_k \in [0,\infty)$, the $n$-component of $\xi^{k}\langle\lambda_0 \mid \dots \mid \lambda_k\rangle$ is given by
      \[
      \xi^{k}_n\langle\lambda_0 \mid \dots \mid \lambda_k\rangle = \sum_{i=1}^{N(k,n)} c^{k}_{n,i}\tau^{k}_{n,i}(\,\cdot\,,\lambda_0,\dots,\lambda_k).
      \]

      \item For any $k\ge0$ and $\lambda_0,\dots,\lambda_k\ge0$,
  \begin{align}\label{equation: xi^k relation}  D^{\prod}(\xi^{k}\langle\lambda_0\mid\dots\mid\lambda_k\rangle) &= \sum_{i=0}^{k-1} (-1)^i \xi^{k-1}\langle\lambda_0\mid\dots\mid\lambda_{i}+\lambda_{i+1}\mid\dots\mid\lambda_k\rangle \\ \nonumber
     &\hspace{0.5em} +  \sum_{i=0}^{k-1} (-1)^{i-1}\xi^{i}\langle\lambda_0\mid\dots\mid\lambda_i\rangle \circ \xi^{k-1-i}\langle\lambda_{i+1}\mid\dots\mid\lambda_k\rangle .
  \end{align}
    \end{enumerate}
\end{lemma}

\begin{proof}
We prove the lemma by induction on $k$.
Once a choice of $\xi^{k}$ has been made for some $k$, it will remain fixed in all subsequent steps.

First, for $k=0$, 
$\xi^0\langle \lambda_0 \rangle = (s_{\lambda_0})_*(\tilde{v})$ depends continuously on $\lambda_0$, and satisfies
\[
D^{\prod}(\xi^0\langle \lambda_0 \rangle) = (s_{\lambda_0})_*(D^{\prod} \tilde{v}) = 0.
\]

Next, assume the lemma holds for all $k' < k$ with $k \ge 1$, and let
\[
\Xi^k\langle \lambda_0 \mid \dots \mid \lambda_k \rangle
\]
denote the right-hand side of \eqref{equation: xi^k relation}.  
By the inductive hypothesis, a straightforward computation shows that
\[
D^{\prod}(\Xi^k\langle \lambda_0 \mid \dots \mid \lambda_k \rangle) = 0,
\]
so $\Xi^k\langle \lambda_0 \mid \dots \mid \lambda_k \rangle$ is a cycle.

We now construct a bounding chain of $\Xi^k\langle \lambda_0 \mid \dots \mid \lambda_k \rangle$ that depends continuously on $\lambda_0,\dots,\lambda_k$. 
For $T = \lambda_0 + \dots + \lambda_k$, consider the maps
\[
h_n \colon [0,1]\times \mathbb{\Delta}^n_T  \to \mathbb{\Delta}^n_T, \quad
(s,t_1, \dots, t_n) \mapsto (s t_1, \dots, s t_n), \quad n \ge 0.
\]
These define deformation retractions of $\mathbb{\Delta}^n_T$ onto the point $\mathbb{\Delta}^n_0 \subset \mathbb{\Delta}^n_T$, compatible with the cosimplicial structure maps.

Define a linear homotopy operator
\begin{align*}
H = \{H_n\}_{n \ge 0} \colon \mathrm{Tot}^{\prod}_\bullet(\{ C^\square(\mathbb{\Delta}^n_T) \}_{n \ge 0})
&\to \mathrm{Tot}^{\prod}_{\bullet+1}(\{ C^\square(\mathbb{\Delta}^n_T) \}_{n \ge 0})\\
(x_n)_{n \ge 0} &\mapsto \big( (h_n)_*(\mathrm{id}_{[0,1]}\times x_n) \big)_{n \ge 0}.
\end{align*}
Then
\[
D^{\prod} \circ H + H \circ D^{\prod} = \mathrm{id} - \mathrm{c},
\]
where $\mathrm{c} = \{ \mathrm{c}_n \}_{n \ge 0}$ is induced by the contraction maps $h_n(\cdot,0)$. Thus,
\[
D^{\prod}\big(H(\Xi^k\langle \lambda_0 \mid \dots \mid \lambda_k \rangle)\big) 
= \Xi^k\langle \lambda_0 \mid \dots \mid \lambda_k \rangle - \mathrm{c}(\Xi^k\langle \lambda_0 \mid \dots \mid \lambda_k \rangle).
\]
It remains to show that 
\[
\mathrm{c}(\Xi^k\langle \lambda_0 \mid \dots \mid \lambda_k \rangle) = 0;
\] then $H(\Xi^k\langle \lambda_0 \mid \dots \mid \lambda_k \rangle)$ gives the desired bounding chain. 

Since $h_n(\cdot,0)$ factors through $\mathbb{\Delta}^n_0$, we can view each component
\[
\mathrm{c}_n(\Xi^k_n\langle\lambda_0\mid\dots\mid\lambda_k\rangle),\quad n\ge0,
\]
of the $(k-1)$-chain $\mathrm{c}(\Xi^k\langle\lambda_0\mid\dots\mid\lambda_k\rangle)$ as lying in $C^{\square}_{\bullet}(\mathbb{\Delta}^n_0)=C^{\square}_{\bullet}(\mathrm{pt})$, which vanishes in positive degrees. Since each such component has degree $k-1+n$, it vanishes automatically if $k>1$ or $n>0$; for the exceptional case $k=1$ and $n=0$, we have 
\begin{align*}
    \mathrm{c}_0(\Xi^1_0\langle\lambda_0\mid\lambda_1\rangle) & = \mathrm{c}_0(\xi^0_0\langle\lambda_0+\lambda_1\rangle-\xi^0_0\langle\lambda_0\rangle\circ\xi^0_0\langle\lambda_1\rangle)\\
    & = \mathrm{c}_0((s_{\lambda_0+\lambda_1})_*(v_0)) - \mathrm{c}_0((s_{\lambda_0})_*(v_0)\circ (s_{\lambda_1})_*(v_0)),
\end{align*}
which is the difference of two terms both equal to the fundamental cycle of $\mathbb{\Delta}^0_0=\mathrm{pt}$, and is therefore zero. This completes the proof.
\end{proof}

\begin{proof}[Proof of Lemma \ref{lemma: J_1 extends A infinity}]
   Fix once and for all the family $\{\xi^k\langle\lambda_0\mid\dots\mid\lambda_n\rangle\}$, or equivalently, a collection of scalars and continuous maps
   \[
   \big\{c^{k}_{n,i} \in R,\,\tau^{k}_{n,i}\colon I^{k+n}\times[0,\infty)^{k+1}\to \mathbb{R}^n\big\}_{k\ge0,\,n\ge0,\, 1\le i\le N(k,n)},
   \]
   provided by Lemma \ref{lemma: xi^k}. 
   
   Let $k\ge1$, and consider points $a_0,\dots,a_k\in X$ together with singular cubes 
   \[
       \sigma_j\colon I^{d_j}\to\mathsf{P}X(a_{j-1},a_j),\quad
       x \mapsto (\gamma_j(x),T_j(x)),  \quad 1\le j\le k,
   \]
   where $T_j\colon I^{d_j}\to[0,\infty)$ and $\gamma_j(x)\colon [0,T_j(x)]\to X$ is a path. 
   
   For any $n\ge0$ and $1\le i\le N(k-1,n)$, define a continuous map
   \begin{align*}
       J_{k,n,i}&\colon I^{d_1}\times\dots\times I^{d_k}\times I^{k-1+n}  \to \mathsf{P}^nX(a_0,a_k) \\
       &(x_1,\dots,x_k,y) 
       \mapsto \big(\sigma_1(x_1)*\dots*\sigma_k(x_k),\,\tau^{k-1}_{n,i}(y,T_1(x_1),\dots, T_k(x_k))\big).
   \end{align*}
   Each $J_{k,n,i}$ determines a normalized cubical chain on $\mathsf{P}^nX(a_0,a_k)$. 
   
   We then define, for every $k\ge1$, an $R$-linear map
   \begin{align*}
   \mathcal{J}'_{X,k}\colon \bigotimes_{1 \le j \le k}\mathsf{C}^{\square}(\mathsf{P}X)(a_{j-1}, a_j) &\to \mathrm{Tot}^{\prod}\big(\{ C^{\square}(\mathsf{P}^n X(a_0,a_k)) \}_{n \ge 0} \big) \\
       \mathcal{J}'_{X,k}(\sigma_1\otimes\dots\otimes \sigma_k)
   &=\bigg(\,\sum_{i=1}^{N(k-1,n)}(-1)^{(k-1)(d_1+\dots+d_k)}c^{k-1}_{n,i}J_{k,n,i}\bigg)_{n\ge0},
   \end{align*}
   and finally define
   \[   \mathcal{J}_{X,k}=\eta^{\square,\mathbb{\Delta}}\circ\mathcal{J}'_{X,k}.
   \]  
   By construction, $\{\mathcal{J}_{X,k}\}_{k\ge1}$ is an $A_\infty$-functor extending $\mathcal{J}_{X,1}=\eta^{\square,\mathbb{\Delta}}\circ q_*\circ\Phi_v$.
\end{proof}

\begin{remark}
There is a subtlety: the collection $\{\mathcal{J}'_{X,k}\}_{k\ge1}$ is \emph{not} an $A_\infty$-functor extending $\mathcal{J}'_{X,1}= q_*\circ\Phi_v$, as the $A_\infty$-functor relations hold only up to switching cube coordinates in normalized cubical singular chains. 
To illustrate, consider the relation between $\mathcal{J}'_{X,1}$ and $\mathcal{J}'_{X,2}$. 
For $\sigma_i\colon I^{d_i}\to \mathsf{P}X(a_{i-1},a_i)$, one needs to compare
\[
D^{\prod}(\mathcal{J}'_{X,2}(\sigma_1\otimes\sigma_2)) + \mathcal{J}'_{X,2}(\partial(\sigma_1\otimes\sigma_2))
\]
with
\[
\mathcal{J}'_{X,1}(\sigma_1\circ\sigma_2) - \mathcal{J}'_{X,1}(\sigma_1)\circ\mathcal{J}'_{X,1}(\sigma_2).
\]
For any $n\ge0$, the $C^{\square}(\mathsf{P}^nX(a_0,a_2))$-component of $\mathcal{J}'_{X,1}(\sigma_1)\circ\mathcal{J}'_{X,1}(\sigma_2)$ is a weighted sum, over $n_1+n_2=n$ and $1\le i\le N(0,n)$, of chains of the form
\begin{align*}
&I^{d_1}\times I^{n_1}\times I^{d_2}\times I^{n_2} \to \mathsf{P}^nX(a_0,a_2),\\
&(x_1,y_1,x_2,y_2) \mapsto \Big(\sigma_1(x_1)*\sigma_2(x_2),\,
\tau^0_{n,i}(y_1,T_1(x_1)),\,T_1(x_1)+\tau^0_{n,i}(y_2,T_2(x_2))\Big).
\end{align*}
The corresponding part in 
$D^{\prod}(\mathcal{J}'_{X,2}(\sigma_1\otimes\sigma_2)) + \mathcal{J}'_{X,2}(\partial(\sigma_1\otimes\sigma_2))$
is, however, a weighted sum of chains of the form
\begin{align*}
&I^{d_1}\times I^{d_2}\times I^{n_1}\times I^{n_2} \to \mathsf{P}^nX(a_0,a_2),\\
&(x_1,x_2,y_1,y_2) \mapsto \Big(\sigma_1(x_1)*\sigma_2(x_2),\,
\tau^0_{n,i}(y_1,T_1(x_1)),\,T_1(x_1)+\tau^0_{n,i}(y_2,T_2(x_2))\Big).
\end{align*}
Thus, the two expressions differ by switching the $I^{d_2}$ and $I^{n_1}$ coordinates, with signs changing accordingly. 
This difference vanishes upon passing to normalized simplicial singular chains, since $\eta^{\square,\mathbb{\Delta}}$ is defined by a sum over all permutations of the cube coordinates, each with its signature.
\end{remark}

\begin{remark}\label{remark: unital A infinity functor}
By construction, $\xi^k\langle \lambda_0 \mid \dots \mid \lambda_k \rangle = 0$ if $k>0$ and some $\lambda_i = 0$.  
For $k=0$, we have $\xi^0_n\langle 0 \rangle = 0$ ($n>0$) and $\xi^0_0 = \mathrm{id}_{\mathrm{pt}} \in C^{\square}_0(\mathbb{\Delta}^0_0) = C^{\square}_0(\mathrm{pt})$.  
Consequently, $\mathcal{J}_X$ is a unital $A_\infty$-functor, in the sense that $\mathcal{J}_{X,1}$ respects identity morphisms and $\mathcal{J}_{X,k}$ ($k>1$) vanishes whenever one of its inputs is an identity morphism.  
It follows that the $A_\infty$-functors $\mathcal{I}_X$ and $\mathcal{F}_X$ are also unital.
\end{remark}

\subsection{Proof of Theorem \ref{theorem: A infinity homotopy}} We need the following acyclicity lemma.

\begin{lemma} \label{lemma: pr_0 quasi-isom}
    For any contractible topological space $X$ and $a,b\in X$, the projection chain map
    \[
        \mathrm{pr}_0 \colon \mathsf{Cobar}^{\prod}(C(X))(a,b) \to R_a\otimes R_b,\qquad
        (x_n)_{n\ge0} \mapsto x_0
    \]  
    is a quasi-isomorphism. Hence, $H_{\bullet}(\mathsf{Cobar}^{\prod}(C(X))(a,b)) \cong R$ is concentrated in degree zero.
\end{lemma}
\begin{proof}
    This is an easy consequence of \cite[Lemma 8.3]{Irie}.
\end{proof}

Now we begin the proof of Theorem \ref{theorem: A infinity homotopy}. Explicitly, we need to show that there exists, for any topological space $X$, $k\ge1$ and $a_0,\dots,a_k\in X$, a natural linear map
\[
\mathcal{H}_{X,k} \colon 
\bigotimes_{1 \le i \le k} \mathsf{Cobar}^{\boxtimes}(\mathcal{C}(X))(a_{i-1},a_i) 
\to
\mathsf{Cobar}^{\prod}(C(X))(a_0,a_k),
\]
and the collection $\{\mathcal{H}_{X,k}\}_{k\ge1}$ satisfies, for any $k \ge 1$,
\begin{align}\nonumber
\mathcal{F}_{X,k} - \mathcal{G}_{X,k}
= \; & D^{\prod} \circ \mathcal{H}_{X,k} 
+ \sum_{p+r = k-1} (-1)^{k-1} \, \mathcal{H}_{X,k} \circ 
\big(1^{\otimes p} \otimes D^{\boxtimes} \otimes 1^{\otimes r}\big) \\ \label{equation: theorem A inifity homotopy}
& + \sum_{p+r = k-2} (-1)^p \, \mathcal{H}_{X,k-1} \circ 
\big(1^{\otimes p} \otimes \mu^{\boxtimes} \otimes 1^{\otimes r}\big) \\ \nonumber
& + \sum_{i + j = k} (-1)^{i}
\Big(
\mu^{\prod} \circ \bigl(\mathcal{H}_{X,i} \otimes \mathcal{F}_{X,j}\bigr)
+ \mu^{\prod} \circ \big(\mathcal{G}_{X,i} \otimes \mathcal{H}_{X,j}\big)
\Big),
\end{align}
where $\mu^{\boxtimes}$ and $\mu^{\prod}$ denote the respective compositions in $\mathsf{Cobar}^{\boxtimes}$ and $\mathsf{Cobar}^{\prod}$.

Consider the following statement $\mathrm{S}(k,d)$ for integers $k\ge 1$ and $d\ge 0$:
\begin{itemize}
    \item $\mathrm{S}(k,d)$: For any topological space $X$ and $a_0,\dots,a_k\in X$, the natural linear maps 
    $\mathcal{H}_{X,k'}$ can be defined on all tensors
    \[
    x_1\otimes\cdots\otimes x_{k'}\in\bigotimes_{i=1}^{k'}\mathsf{Cobar}^{\boxtimes}(\mathcal{C}(X))(a_{i-1},a_i)
    \]
    with $k'<k$ and $\deg x_1+\cdots+\deg x_{k'}\le d$, 
    in such a way that \eqref{equation: theorem A inifity homotopy} holds on these elements.
\end{itemize}
Theorem \ref{theorem: A infinity homotopy} is equivalent to the assertion that $\mathrm{S}(k,d)$ holds for all $k\ge 1$ and $d\ge 0$. 

First of all, we set $\mathcal{H}_{X,k}$ to be zero if any of its inputs is a multiple of an identity morphism; then \eqref{equation: theorem A inifity homotopy} holds on such inputs by Remark~\ref{remark: unital A infinity functor} and the fact that $\mathcal{G}_X$ respects identities.

Next, we prove $\mathrm{S}(k,d)$ for all $k,d$ by induction on $(k,d)$ in lexicographic order:
\[
(k',d') < (k,d) \quad \text{if either (i) $k'<k$, or (ii) $k'=k$ and $d'<d$}.
\]
For the induction, we denote by $\mathcal{H}^d_{X,k}$ the restriction of the (yet to be constructed) map 
$\mathcal{H}_{X,k}$ to elements of total internal degree $\le d$; similarly, we write $\mathcal{F}^d_{X,k}$ and $\mathcal{G}^d_X$ for the corresponding restrictions of $\mathcal{F}_{X,k}$ and $\mathcal{G}_X$. 
Once defined for a pair $(k,d)$, the maps $\mathcal{H}^{d}_{X,k}$ are not modified in any subsequent steps.

Fix $(k,d)$ with $k\ge0$ and $d\ge0$, suppose $\mathrm{S}(k',d')$ holds for all $(k',d')<(k,d)$, we aim to construct 
$\mathcal{H}^{d}_{X,k}$ with the desired properties. (Note that the assumption for $(k,d)=(1,0)$ is vacuumly true.) It suffices to define 
\[
\mathcal{H}^{d}_{X,k}(x_1\otimes\dots\otimes x_k)
\]
for all 
\[
x_i\in R_{a_{i-1}}\boxtimes(s^{-1}\overline{\mathcal{C}(X)})^{\boxtimes n_i}\boxtimes R_{a_i}, \; n_i>0 \; (1\le i\le k)
\]
with 
\[
\deg x_1+\dots+\deg x_k = d
\]
in a way that is natural in $X$, and to verify that on such $x_1\otimes\dots\otimes x_k$,
\begin{align}\nonumber
D^{\prod} \circ \mathcal{H}^d_{X,k}
= \; & \mathcal{F}^d_{X,k} - \mathcal{G}^d_{X,k}
+ \sum_{p+r = k-1} (-1)^{k} \, \mathcal{H}^{d-1}_{X,k} \circ 
\big(1^{\otimes p} \otimes D^{\boxtimes} \otimes 1^{\otimes r}\big) \\ \label{equation: H^d_X,k}
& + \sum_{p+r = k-2} (-1)^{p-1} \, \mathcal{H}^d_{X,k-1} \circ 
\big(1^{\otimes p} \otimes \mu^{\boxtimes} \otimes 1^{\otimes r}\big) \\ \nonumber
& + \sum_{i + j = k} (-1)^{i-1}
\Big(
\mu^{\prod} \circ \bigl(\mathcal{H}^d_{X,i} \otimes \mathcal{F}^d_{X,j}\bigr)
+ \mu^{\prod} \circ \big(\mathcal{G}^d_{X,i} \otimes \mathcal{H}^d_{X,j}\big)
\Big).
\end{align}

For any ordered $k$-tuple $\vec{\boldsymbol{d}}$ of ordered tuples of positive integers, $\vec{\boldsymbol{d}}=(\boldsymbol{d}_1,\dots,\boldsymbol{d}_k)$ with $\boldsymbol{d}_i=(d_{i,1},\dots,d_{i,n_i})$ and $d_{i,j}>0$, 
define its degree by
\[
\deg \vec{\boldsymbol{d}}\coloneq \sum_{i=1}^k\sum_{j=1}^{n_i}(d_{i,j}-1).
\]
Let
\[
\iota_{i,j}\colon \mathbb{\Delta}^{d_{i,j}}\hookrightarrow \bigvee_{i'=1}^k\bigvee_{j'=1}^{n_{i'}}\mathbb{\Delta}^{d_{i',j'}}\eqqcolon \mathbb{\Delta}^{\vec{\boldsymbol{d}}},\quad\quad 1\le i\le k,\, 1\le j\le n_i
\]
be the obvious inclusion maps, where the wedge sum denotes a \textit{$k$-necklace}, i.e., it is formed by successively identifying the last vertex of each simplex with the first vertex of the next in the order \[\mathbb{\Delta}^{d_{1,1}} \vee \cdots \vee \mathbb{\Delta}^{d_{1,n_1}}\vee \mathbb{\Delta}^{d_{2,1}} \vee \cdots \vee \mathbb{\Delta}^{d_{k,n_k}}.\] There is a canonical chain
\[
\otimes_{i=1}^k\boxtimes_{j=1}^{n_i} s^{-1}\iota_{i,j}\in \bigotimes_{i=1}^k\Big(R_{v_{m_{i-1}}}\boxtimes\Big(\bigboxtimes_{j=1}^{n_i} s^{-1}\mathcal{C}_{d_{i,j}}(\mathbb{\Delta}^{\vec{\boldsymbol{d}}})\Big) \boxtimes R_{v_{m_i}}\Big),
\]
where $m_i=\sum_{i'=1}^i\sum_{j=1}^{n_{i'}} d_{i',j}$, and $v_{p}$ denotes the $(p+1)$-th vertex of $\mathbb{\Delta}^{\vec{\boldsymbol{d}}}$.

Now assume $\deg\vec{\boldsymbol{d}}=d$. Denote by 
\[
\Psi^{\vec{\boldsymbol{d}}}\in\mathsf{Cobar}^{\prod}(C(\mathbb{\Delta}^{\vec{\boldsymbol{d}}}))(v_0,v_d)
\]
the chain obtained by applying the right-hand side of \eqref{equation: H^d_X,k} to $\otimes_{i=1}^k\boxtimes_{j=1}^{n_i} s^{-1}\iota_{i,j}$. Then we have the following lemma.

\begin{lemma}\label{lemma: homotopy Psi^d}
    There exists a chain $\psi^{\vec{\boldsymbol{d}}}\in\mathsf{Cobar}^{\prod}(C(\mathbb{\Delta}^{\vec{\boldsymbol{d}}}))(v_0,v_d)$ such that
\[
D^{\prod}\psi^{\vec{\boldsymbol{d}}}=\Psi^{\vec{\boldsymbol{d}}}.
\]
\end{lemma}
\begin{proof}
     One checks by the inductive hypothesis that $\Psi^{\vec{\boldsymbol{d}}}$ is a cycle, so it remains to show that $[\Psi^{\vec{\boldsymbol{d}}}]=0$ in homology. 
     
     Notice that $\deg\Psi^{\vec{\boldsymbol{d}}}=k-1+\sum_{i,j}(d_{i,j}-1)$. If $k>1$ or some $d_{i,j}>1$, then 
     $\deg\Psi^{\vec{\boldsymbol{d}}}>0$, and then
     by Lemma~\ref{lemma: pr_0 quasi-isom}, $[\Psi^{\vec{\boldsymbol{d}}}]=0$. 
     
     It remains to consider the case $k=1$ and $\vec{\boldsymbol{d}}=\boldsymbol{d}_1=(1,\dots,1)$. For simplicity, denote by     \[\iota_i\colon\mathbb{\Delta}^1\hookrightarrow\mathbb{\Delta}^{1}\vee\dots\vee\mathbb{\Delta}^{1}=\mathbb{\Delta}^{1,\dots,1}\] the inclusion map sending $\mathbb{\Delta}^1$ onto the $i$-th copy of $\mathbb{\Delta}^1$.    
     By definition, 
     \[
     \Psi^{1,\dots,1}=\mathcal{F}_{\mathbb{\Delta}^{1,\dots,1},1}(s^{-1}\iota_{1}\boxtimes\dots\boxtimes s^{-1}\iota_{k})-\mathcal{G}_{\mathbb{\Delta}^{1,\dots,1}}(s^{-1}\iota_{1}\boxtimes\dots\boxtimes s^{-1}\iota_{k}).
     \]
     Since $\mathcal{F}_X$, $\mathcal{G}_X$ are functorial in homology, 
     \[
     [\Psi^{1,\dots,1}]=\Pi_{i=1}^k[\mathcal{F}_{\mathbb{\Delta}^{1,\dots,1},1}(s^{-1}\iota_{i})]-\Pi_{i=1}^k[\mathcal{G}_{\mathbb{\Delta}^{1,\dots,1}}(s^{-1}\iota_{i})].
     \] 
     Since $\mathcal{F}=\mathcal{I}\circ\mathcal{A}$, for each $1\le i\le k$, we have
     \[\mathcal{F}_{\mathbb{\Delta}^{1,\dots,1},1}(s^{-1}\iota_{i}) = \mathcal{I}_{\mathbb{\Delta}^{1,\dots,1},1}(\bar{\iota_i})=\big( (v_{i-1}\cdot 1_R)\otimes (s^{-1}\overset{\leftarrow}{\iota_i})^{\otimes n}\otimes(1_R\cdot v_i)\big)_{n\ge0},
     \]
     where $\bar{\iota_i}\in\mathsf{P}\mathbb{\Delta}^{1,\dots,1}(v_{i-1},v_i)$ denotes the point representing the path $\iota_i$, and $\overset{\leftarrow}{\iota_i}$ is the reversal of $\iota_i$, i.e. $\overset{\leftarrow}{\iota_i}(t)=\iota_i(1-t)$ for $t\in[0,1]$. Also,
     \[\mathcal{G}_{\mathbb{\Delta}^{1,\dots,1}}(s^{-1}\iota_i)  = v_{i-1}\cdot 1_R\cdot v_i + (v_0\cdot 1_R)\otimes(s^{-1}\iota_i)\otimes (1_R\cdot v_1).\]
     It follows that 
     \[
     (\mathrm{pr}_0)_*([\Psi^{1,\dots,1}])=[1_R]-[1_R]=0\in H_0(\mathbb{\Delta}^{1,\dots,1}).
     \] 
     By Lemma~\ref{lemma: pr_0 quasi-isom}, $[\Psi^{1,\dots,1}]=0$, and the proof is complete.
\end{proof}

Fix a choice of $\psi^{\vec{\boldsymbol{d}}}\in\mathsf{Cobar}^{\prod}(C(\mathbb{\Delta}^{\vec{\boldsymbol{d}}}))(v_0,v_d)$ provided by Lemma \ref{lemma: homotopy Psi^d}.
For any topological space $X$, integer $k\ge1$ and any two-indexed family of simplices
\[
\sigma_{i,j}\colon \Delta^{d_{i,j}}\to X,\quad 1\le i\le k,\quad 1\le j\le n_i,
\]
such that the image of the last vertex of each simplex coincides with the image of the first vertex of the next one, we denote by
\begin{equation}\label{equation: multi necklace}
\vec{\boldsymbol{\sigma}}
=\bigvee_{i=1}^k\bigvee_{j=1}^{n_i}\sigma_{i,j}\colon \Delta^{\vec{\boldsymbol{d}}}\to X
\end{equation}
the resulting map, which we call a \textit{$k$-necklace in $X$}. Then, we define
\[
\mathcal{H}^{d}_{X,k}(\otimes_{i=1}^k\boxtimes_{j=1}^{n_i} s^{-1}\sigma_{i,j}) = \vec{\boldsymbol{\sigma}}_*(\psi^{\vec{\boldsymbol{d}}}).
\]
Clearly $\mathcal{H}^{d}_X$ is natural in $X$, and by a routine computation using the inductive hypothesis, equation \eqref{equation: H^d_X,k} holds. This completes the proof of Theorem \ref{theorem: A infinity homotopy}. \qed

\subsection{Elaboration on Remark~\ref{remark: Chen Adams inverse simply connected}}
\label{subsection: Chen Adams inverse simply connected}
For any topological space $X$ with a basepoint $b \in X$, recall that $C^1(X,b) \subset C(X)$ denotes the dg subcoalgebra generated by 1-reduced simplices at $b$, i.e., simplices whose edges all map to $b$. 
Define the categorical subcoalgebra $\mathcal{C}^1(X,b) \subset \mathcal{C}(X)$ by setting $\mathcal{C}^1(X,b) = C^1(X,b)$ as $R$-modules and $\mathcal{S}(\mathcal{C}^1(X,b)) = \{b\}$. 
It is straightforward to check that $\mathcal{C}^1(X,b)$ is a categorical coalgebra with strict differential and zero curvature, hence a usual dg coalgebra, and that the linear isomorphism 
\[
\mathcal{C}^1(X,b) \cong C^1(X,b)
\]
is an isomorphism of coalgebras in the usual sense.

Let us now examine the restriction of the $A_\infty$-functor $\mathcal{F}_X$ and the functor $\mathcal{G}_X$ to the dg subcategory $\mathsf{Cobar}^{\boxtimes}(\mathcal{C}^1(X,b))$ of $\mathsf{Cobar}^{\boxtimes}(\mathcal{C}(X))$. We denote these restrictions by
$\mathcal{F}_X|_{\mathcal{C}^1}$ and $\mathcal{G}_X|_{\mathcal{C}^1}$, respectively.

First, we claim that if $u_n\in C_n(\mathbb{\Delta}^n)$ is chosen as $u_n=\mathrm{id}_{\mathbb{\Delta}^n}$ for all $n\ge0$ in the construction of $\mathcal{I}_1$, then
\[
\mathcal{F}_X|_{\mathcal{C}^1}\colon\mathsf{Cobar}^{\boxtimes}(\mathcal{C}^1(X,b))\to\mathsf{Cobar}^{\prod}(C(X))
\]
factors through the inclusion of dg categories 
\[
\mathsf{Cobar}^{\prod}(C^1(X,b))\subset\mathsf{Cobar}^{\prod}(C(X)).
\]

It suffices to check that for any $k\ge1$, 1-reduced simplices $\sigma_m\colon \mathbb{\Delta}^{d_m}\to X$ ($1\le m\le k$), $n\ge0$, $1\le j\le n$ and $1\le i\le N(k-1,n)$, the restriction of the map
\begin{align*}
    I^{d_1-1}\times\dots\times I^{d_k-1}&\times I^{k-1+n} \to X \\
    (x_1,\dots,x_k,y) &\mapsto \\ \mathrm{ev}_{n,j}&\big(A_X(\sigma_1)(x_1)*\dots*A_X(\sigma_k)(x_k),\tau^{k-1}_{n,i}(y,T_1(x_1),\dots,T_k(x_k))\big)
\end{align*}
to any edge of the domain cube is the constant map to $b$. Here, 
\[
\mathrm{ev}_{n,j}\colon \mathsf{P}^nX(b,b)\to X
\]
is the evaluation map at the $j$-th marked point, and $\tau^{k-1}_{n,i}$ is constructed in Lemma~\ref{lemma: xi^k}.

This can be checked as follows. There are two types of edges:
\begin{itemize}
    \item A vertex in $I^{d_1-1}\times\dots\times I^{d_k-1}$ with an edge in $I^{k-1+n}$. By the definition of Adams' map, such an edge corresponds to the concatenation of some edges of the 1-reduced simplices $\sigma_1,\dots,\sigma_k$, hence is mapped to $b$.
    \item An edge in $I^{d_1-1}\times\dots\times I^{d_k-1}$ with a vertex in $I^{k-1+n}$. By the choice of $u_n$ and the explicit inductive construction of $\xi^{k-1}_{n}$ in Lemma~\ref{lemma: xi^k}, the point $\tau^{k-1}_{n,i}(y,T_1,\dots,T_k)\in\mathbb{\Delta}^n_{T_1+\dots+T_k}$ is a vertex of the scaled simplex $\mathbb{\Delta}^n_{T_1+\dots+T_k}$ whenever $y\in I^{k-1+n}$ is a vertex. Therefore such an edge is mapped to one of the endpoints of $A_X(\sigma_1)(x_1)*\dots*A_X(\sigma_k)(x_k)$, namely the basepoint $b$.
\end{itemize}
Hence the claim follows.

Next, observe that the map \eqref{equation: e C(X) to R} vanishes on $C^1(X,b)$, so
\[
\mathcal{G}_X|_{\mathcal{C}^1}\colon \mathsf{Cobar}^{\boxtimes}(\mathcal{C}^1(X,b))
  \to \mathsf{Cobar}^{\prod}(C(X))
\]
 is simply given by the inclusion induced from
\[
\overline{\mathcal{C}^1(X,b)}\cong\overline{C^1(X,b)}\hookrightarrow C^1(X,b)\hookrightarrow C(X)
\]
and the inclusion of a direct sum into a direct product.
Moreover, since $s^{-1}\overline{C^1(X,b)}$ is concentrated in positive degrees, 
\[
\bigoplus_{n=0}^\infty R_b\otimes (s^{-1}\overline{C^1(X,b)})^{\otimes n}\otimes R_b = \prod_{n=0}^\infty R_b\otimes (s^{-1}\overline{C^1(X,b)})^{\otimes n}\otimes R_b.
\]
Hence $\mathcal{G}_X$ identifies
$\mathsf{Cobar}^{\boxtimes}(\mathcal{C}^1(X,b))=\mathsf{Cobar}(C^1(X,b))$ with the conormalization of $\mathsf{Cobar}^{\prod}(C^1(X,b))$, viewed as a subcategory of $\mathsf{Cobar}^{\prod}(C(X))$.

Now assume $X$ is simply connected. By the results in Section \ref{section: Adams general}, the inclusion
\[
\mathsf{Cobar}^{\boxtimes}(\mathcal{C}^1(X,b)) \hookrightarrow \mathsf{Cobar}^{\boxtimes}(\mathcal{C}(X))
\]
is a quasi-equivalence. By comparing spectral sequences, the inclusion
\[
\mathsf{Cobar}^{\prod}(C^1(X,b)) \hookrightarrow \mathsf{Cobar}^{\prod}(C(X))
\]
is also a quasi-equivalence. The discussion above is summarized as follows.

\begin{proposition}
For any simply connected topological space $X$, the $A_\infty$-functor
\[\mathcal{I}_X \colon \mathsf{C}^{\square}(\mathsf{P}X) \to \mathsf{Cobar}^{\prod}(C(X))\]
acts as a left $A_\infty$-homotopy inverse of the functor
\[\mathcal{A}_X \colon \mathsf{Cobar}^{\boxtimes}(\mathcal{C}(X)) \to \mathsf{C}^{\square}(\mathsf{P}X)\]
on the common subcategory 
$\mathsf{Cobar}^{\boxtimes}(\mathcal{C}^1(X,b))$.
More precisely, the inclusions of this subcategory into $\mathsf{Cobar}^{\boxtimes}(\mathcal{C}(X))$   and $\mathsf{Cobar}^{\prod}(C(X))$ are both quasi-equivalences, and the restriction of the composition $\mathcal{I}_X \circ \mathcal{A}_X$ to this subcategory is $A_\infty$-homotopic to the conormalization inclusion
$\mathsf{Cobar}^{\boxtimes}(\mathcal{C}^1(X,b)) \hookrightarrow \mathsf{Cobar}^{\prod}(C^1(X,b))$.
\end{proposition}
\begin{remark}
    One could alternatively \emph{define} $\mathsf{Cobar}^{\prod}(C)$ by additionally modding out the ideal generated by $C_0$. On $\mathsf{Cobar}^{\prod}(C^1(X,b))$, this quotient is naturally isomorphic to the conormalization, so in the alternative definition, the map $\mathsf{Cobar}^{\boxtimes}(\mathcal{C}^1(X,b)) \hookrightarrow \mathsf{Cobar}^{\prod}(C^1(X,b))$ is the identity.
\end{remark}

\section{Adams' map and the universal cover}
\label{section: Adams general}
In this section, we give an elementary proof of Theorem~\ref{theorem: Adams} and Corollary~\ref{corollary: pointed Adams} by passing to the universal cover and using Adams' classical cobar theorem for simply connected spaces. This section is essentially self-contained and does not use any of the results proved earlier. Corollary~\ref{corollary: pointed Adams} was originally proved in \cite{RZ_cubical} using different methods. 

Let $X$ be a topological space. We may assume $X$ is path-connected; otherwise we argue separately on each path-connected component of $X$. 
Fix once and for all, for each pair $(a,b)\in X\times X$ with $a\neq b$, a unit-length path
\[
\gamma_{ab}\colon[0,1]\to X,\qquad \gamma_{ab}(0)=a,\ \gamma_{ab}(1)=b,
\]
and require that $\gamma_{ba}(t)=\gamma_{ab}(1-t)$ for all $a,b\in X$ with $a\neq b$ and $t\in[0,1]$.  
For each $b\in X$, we also define a zero-length path 
\[
\gamma_{bb}\colon[0,0]\to X, \qquad 
\gamma_{bb}(0)=b.
\]
Denote the resulting family of paths by 
\begin{equation}\label{equaiton: O_X choice of paths}
    \mathcal{O}_X=\{\gamma_{ab}\}_{a,b\in X}.
\end{equation}

For any two $a,b\in X$, there is a homotopy equivalence
\[
L_{\gamma_{ab}}\colon \mathsf{P}X(b,b)\to \mathsf{P}X(a,b),\qquad \gamma\mapsto \gamma_{ab}*\gamma,
\]
with homotopy inverse $L_{\gamma_{ba}}\colon \gamma'\mapsto \gamma_{ba}*\gamma'$. Similarly, for $a\neq b$, define an $R$-linear map
\[
L^{\boxtimes}_{\gamma_{ab}}\colon 
\mathsf{Cobar}^{\boxtimes}(\mathcal{C}(X))(b,b)\to \mathsf{Cobar}^{\boxtimes}(\mathcal{C}(X))(a,b),\qquad 
x\mapsto (s^{-1}\gamma_{ab})\boxtimes x,
\]
where $\gamma_{ab}$ is regarded as a $1$-simplex in $X$. Then $L^{\boxtimes}_{\gamma_{ab}}$ is a chain-homotopy equivalence with chain-homotopy inverse $L^{\boxtimes}_{\gamma_{ba}}\colon x'\mapsto (s^{-1}\gamma_{ba})\boxtimes x'$, and a chain homotopy from the identity to $L^{\boxtimes}_{\gamma_{ba}}\circ L^{\boxtimes}_{\gamma_{ab}}$ is given by $x\mapsto (s^{-1}\sigma)\boxtimes x$, where $\sigma=[v_0,v_1,v_2]$ has edges $[v_0,v_1]=\gamma_{ba}$, $[v_1,v_2]=\gamma_{ab}$, and $[v_0,v_2]$ is the degenerate edge given by the constant path at $b$. We also include the case $a=b$ by setting $L^{\boxtimes}_{\gamma_{bb}}=\mathrm{id}$.

Clearly the maps $L_{\gamma_{ab}}$ and $L^{\boxtimes}_{\gamma_{ab}}$ are compatible with Adams' map. More precisely, there is a commutative diagram
\begin{equation}\label{equation: diagram cobar boxtimes change points}
  \begin{tikzcd}
  \mathsf{Cobar}^{\boxtimes}(\mathcal{C}(X))(b,b) \arrow[r, "\mathcal{A}_X"] \arrow[d,"L^{\boxtimes}_{\gamma_{ab}}", "\simeq"'] & C^{\square}(\mathsf{P}X)(b,b) \arrow[d, "(L_{\gamma_{ab}})_*", "\simeq"'] \\		\mathsf{Cobar}^{\boxtimes}(\mathcal{C}(X))(a,b) \arrow[r, "\mathcal{A}_X"] & C^{\square}(\mathsf{P}X)(a,b).
  \end{tikzcd}
\end{equation}

The diagram \eqref{equation: diagram cobar boxtimes change points} implies that Theorem \ref{theorem: Adams} is equivalent to the following proposition, whose proof is postponed to the end of this section.

\begin{proposition}
    \label{proposition: Adams theorem based boxtimes}
    For any pointed path-connected space $(X,b)$,
    \[
    \mathcal{A}_X \colon \mathsf{Cobar}^{\boxtimes}(\mathcal{C}(X))(b,b) \to C^{\square}(\mathsf{P}X)(b,b)
    \]
    is a quasi-isomorphism.
\end{proposition}

For any $b\in X$, consider the dg subcoalgebra $C^0(X,b)$ of $C(X)$ generated by simplices with all vertices at $b$. Let $\mathcal{C}^0(X,b)$ be the categorical subcoalgebra of $\mathcal{C}(X)$ with $\mathcal{C}^0(X,b)=C^0(X,b)$ as $R$-modules and $\mathcal{S}(\mathcal{C}^0(X,b))=\{b\}$. 

Define an $R$-linear map
\begin{align*}
E_{-}\colon s^{-1}\overline{C^0(X,b)} &\longrightarrow
\bigl(R_b\boxtimes s^{-1}\overline{\mathcal{C}^0(X,b)}\boxtimes R_b\bigr)\oplus R_b,\\[2pt]
s^{-1}\sigma &\longmapsto s^{-1}\sigma - e(\sigma),\qquad
\sigma\colon (\Delta^n,\text{vertices})\to (X,b),
\end{align*}
where $e$ is \eqref{equation: e C(X) to R}.  
Extend $E_{-}$ multiplicatively over~$\boxtimes$.  
This induces a linear map
\[
E_{-}\colon \mathsf{Cobar}\bigl(C^0(X,b)\bigr)
 \longrightarrow \mathsf{Cobar}^{\boxtimes}\bigl(\mathcal{C}^0(X,b)\bigr),
\]
which is a chain map by straightforward calculation. Moreover, $E_{-}$ is a dg algebra isomorphism, whose inverse $E_{+}$ is given by $s^{-1}\sigma\mapsto s^{-1}\sigma+e(\sigma)$ on generators.

Denote by 
\[
\iota_b\colon\mathcal{C}^0(X,b)\hookrightarrow \mathcal{C}(X)
\]
the natural inclusion. Using the family $\mathcal{O}_X$ \eqref{equaiton: O_X choice of paths}, one can deform each simplex in $X$ to one with all vertices at $b$, proceeding by induction on the dimension of the simplex. This procedure induces a morphism of categorical coalgebras 
\[
f_{\mathcal{O}_X}\colon \mathcal{C}(X)\to \mathcal{C}^0(X,b)
\] 
satisfying $f_{\mathcal{O}_X}\circ \iota_b = \mathrm{id}_{\mathcal{C}^0(X,b)}$, as well as an $R$-linear map 
$H_{\mathcal{O}_X}\colon \mathcal{C}_\bullet(X)\to \mathcal{C}_{\bullet+1}(X)$ 
serving as a homotopy between $\mathrm{id}_{\mathcal{C}(X)}$ and the morphism $\iota_b\circ f_{\mathcal{O}_X}$. Consequently,
$\mathsf{Cobar}^{\boxtimes}(f_{\mathcal{O}_X})\circ\mathsf{Cobar}^{\boxtimes}(\iota_b)$ is the identity on $\mathsf{Cobar}^{\boxtimes}(\mathcal{C}^0(X,b))(b,b)$,
and the extension of $H_{\mathcal{O}_X}$ to a derivation, 
$
\widehat{H}_{\mathcal{O}_X}\colon \mathsf{Cobar}^{\boxtimes}(\mathcal{C}(X))(b,b)\to \mathsf{Cobar}^{\boxtimes}(\mathcal{C}(X))(b,b),
$
is a chain homotopy between the identity and $\mathsf{Cobar}^{\boxtimes}(\iota_b)\circ \mathsf{Cobar}^{\boxtimes}(f_{\mathcal{O}_X})$. 
Therefore, the natural inclusion 
\[
\mathsf{Cobar}^{\boxtimes}(\iota_b)\colon \mathsf{Cobar}^{\boxtimes}(\mathcal{C}^0(X,b))\hookrightarrow \mathsf{Cobar}^{\boxtimes}(\mathcal{C}(X))(b,b)
\] 
is a chain-homotopy equivalence.

Clearly the maps $E_{-}$ and $\mathsf{Cobar}^{\boxtimes}(\iota_b)$ are compatible with Adams' map, i.e., there is a commutative diagram
\begin{equation}\label{equation: diagram cobar at b different versions}
  \begin{tikzcd}
   \mathsf{Cobar}(C^0(X,b)) \arrow[r,"E_{-}","\cong"'] \arrow[dr, "\mathcal{A}_X"']& \mathsf{Cobar}^{\boxtimes}(\mathcal{C}^0(X,b))(b,b) \arrow[rr, hook, "\mathsf{Cobar}^{\boxtimes}(\iota_b)", "\simeq"'] \arrow[d, "\mathcal{A}_X"] & & \mathsf{Cobar}^{\boxtimes}(\mathcal{C}(X))(b,b) \arrow[dll, "\mathcal{A}_X"]\\
   & C^{\square}(\mathsf{P}X)(b,b). &
  \end{tikzcd}
\end{equation}

The diagram \eqref{equation: diagram cobar at b different versions} implies that Corollary \ref{corollary: pointed Adams} is equivalent to Proposition~\ref{proposition: Adams theorem based boxtimes}.

\begin{remark}\label{remark: Adams original} 
Adams’ original proof \cite{Adams} shows that
\[
\mathcal{A}_X\colon 
\mathsf{Cobar}(C^1(X,b))
  \longrightarrow 
C^{\square}(\mathsf{P}X)(b,b)
\]
is a quasi-isomorphism if $X$ is simply connected. 
Note that the natural inclusion
\[
\mathsf{Cobar}(C^1(X,b))
  \hookrightarrow
\mathsf{Cobar}(C^0(X,b))
\]
is compatible with Adams’ map, and is a quasi-isomorphism provided that $X$ is simply connected. 
Consequently, by 
\eqref{equation: diagram cobar boxtimes change points}\eqref{equation: diagram cobar at b different versions},
\[
\mathcal{A}_X\colon
\mathsf{Cobar}^{\boxtimes}(\mathcal{C}(X))(a,b)
  \longrightarrow 
C^{\square}(\mathsf{P}X)(a,b)
\]
is a quasi-isomorphism for all $a,b\in X$ whenever $X$ is simply connected.
\end{remark}

\begin{proof}[Proof of Proposition~\ref{proposition: Adams theorem based boxtimes}]
For any $b \in X$, denote by $\widetilde{X}_b$ the universal cover of $X$ at $b$ and denote by
\[
\pi_b\colon (\widetilde{X}_b,[c_b])\to (X,b)
\]
be the universal covering map, where $[c_b]$ denotes the homotopy class of the constant path $c_b\colon[0,1]\to X$ at $b$.

First, by lifting paths in $X$ to $\widetilde{X}_b$, we obtain a homeomorphism of spaces
\[
\mathsf{P}X(b,b) \cong \bigsqcup_{\alpha\in\pi_1(X,b)} \mathsf{P}\widetilde{X}_b([c_b],\alpha),
\]
which induces an isomorphism of chain complexes
\begin{equation}\label{equation: cube chain path space direct sum pi_1}
\mathsf{C}^\square(\mathsf{P}X)(b,b) \cong 
C^\square\Big(\bigsqcup_{\alpha\in\pi_1(X,b)} \mathsf{P}\widetilde{X}_b([c_b],\alpha)\Big)
= \bigoplus_{\alpha\in\pi_1(X,b)} \mathsf{C}^\square(\mathsf{P}\widetilde{X}_b)([c_b],\alpha).
\end{equation}

Next, by lifting necklaces in $X$ to $\widetilde{X}_b$, we obtain a chain complex isomorphism
\begin{equation}\label{equation: Cobar boxtimes direct sum pi_1}
\mathsf{Cobar}^{\boxtimes}(\mathcal{C}(X))(b,b)
\cong
\bigoplus_{\alpha\in\pi_1(X,b)} \mathsf{Cobar}^{\boxtimes}(\mathcal{C}(\widetilde{X}_b))([c_b],\alpha).
\end{equation}
More precisely, recall that $\mathbb{\Delta}^{d_1}\vee\dots\vee\mathbb{\Delta}^{d_k}$ is the wedge sum of $\mathbb{\Delta}^{d_1},\dots,\mathbb{\Delta}^{d_k}$ formed by successively identifying the last vertex of each simplex with the first vertex of the next.
Every necklace in $X$
\[
\boldsymbol{\sigma}\colon (\mathbb{\Delta}^{d_1}\vee\dots\vee\mathbb{\Delta}^{d_k},v_0,v_d) \to (X,b,b)
\]
uniquely lifts to a necklace in $\widetilde{X}_b$
\[
\widetilde{\boldsymbol{\sigma}}\colon (\mathbb{\Delta}^{d_1}\vee\dots\vee\mathbb{\Delta}^{d_k},v_0,v_d) \to (\widetilde{X}_b,[c_b],\pi^{-1}(b))
\]
such that $\pi_b\circ\widetilde{\boldsymbol{\sigma}}=\boldsymbol{\sigma}$, where $d=d_1+\dots+d_k$ and $v_i$ denotes the $(i+1)$-th vertex of $\mathbb{\Delta}^{d_1}\vee\dots\vee\mathbb{\Delta}^{d_k}$.  
The endpoint $\widetilde{\boldsymbol{\sigma}}(v_d)$ is uniquely determined by the based homotopy class of 
$\boldsymbol{\sigma}\circ \gamma$
for an arbitrary path $\gamma$ in $\mathbb{\Delta}^{d_1}\vee\dots\vee\mathbb{\Delta}^{d_k}$ from $v_0$ to $v_d$.  
Conversely, every necklace in $\widetilde{X}_b$ arises as the lift of a unique necklace in $X$.  
Thus, for each $k\ge0$ there is an $R$-module isomorphism
\[
R_b \boxtimes (s^{-1}\overline{\mathcal{C}(X)})^{\boxtimes k} \boxtimes R_b
\cong
\bigoplus_{\alpha\in\pi_1(X,b)} 
R_{[c_b]} \boxtimes (s^{-1}\overline{\mathcal{C}(\widetilde{X}_b)})^{\boxtimes k} \boxtimes R_{\alpha},
\]
and taken together, these induce the isomorphism \eqref{equation: Cobar boxtimes direct sum pi_1}.

By Remark \ref{remark: Adams original}, for any $\alpha\in\pi_1(X,b)$,
\[
\mathcal{A}_{\widetilde{X}_b} \colon \mathsf{Cobar}^{\boxtimes}(\mathcal{C}(\widetilde{X}_b))([c_b],\alpha) \to C^{\square}(\mathsf{P}\widetilde{X}_b)([c_b],\alpha)
\]
is a quasi-isomorphism.  
The proof of the proposition is then completed by observing that the isomorphisms \eqref{equation: cube chain path space direct sum pi_1} and \eqref{equation: Cobar boxtimes direct sum pi_1} are compatible with Adams' map.
\end{proof}

\section{Extension of the results to arbitrary spaces} 
\label{section: non simply connected results}
In this section we give the constructions of 
$\mathsf{C^P}(X)$, $\widetilde{It}_X$, $\widetilde{\mathcal{I}}_X$, $\widetilde{\mathcal{F}}_X$, $\widetilde{\mathcal{G}}_X$, and $\widetilde{\mathcal{H}}_X$ for any topological space $X$, thereby proving Theorem~\ref{theorem: A infinity map and homotopy with pi_1}. 
The constructions of $\mathsf{C^P}(X)$ and $\widetilde{It}_X$ are the key ingredients; once these are established, the remaining results follow by arguments parallel to those used for Theorem~\ref{theorem: A infinity map I} and Theorem~\ref{theorem: A infinity homotopy}.

We define the dg category $\mathsf{C^P}(X)$ as follows. The objects are the points of $X$. For any $a,b\in X$, the morphism complex is
\[
\mathsf{C^P}(X)(a,b)
 = \bigoplus_{\beta\in\Pi_1X(a,b)}
    \mathsf{Cobar}^{\prod}(C(\widetilde{X}_a))([c_a],\beta),
\]
where $\Pi_1X$ denotes the fundamental groupoid of $X$, $(\widetilde{X}_a,[c_a])$ is the universal covering space of $(X,a)$, and $\beta\in\Pi_1X(a,b)$ is viewed as a point of $\widetilde{X}_a$ via the canonical inclusion
\[
\Pi_1X(a,b)\hookrightarrow
 \bigcup_{c\in X}\Pi_1X(a,c)=\widetilde{X}_a.
\]
The composition of morphisms is induced by the map
\begin{align*}
    (R_{[c_a]}\otimes(s^{-1}C(\widetilde{X}_a))^{\otimes n}\otimes R_{\beta})\otimes (R_{[c_b]}\otimes(&s^{-1}C(\widetilde{X}_b))^{\otimes n'}\otimes R_{\gamma}) \\
    &\to R_{[c_a]}\otimes (s^{-1}C(\widetilde{X}_a))^{\otimes n+n'}\otimes R_{\beta\circ\gamma}\\
    (1_R\otimes x\otimes 1_R)\otimes (1_R\otimes y\otimes 1_R) &\mapsto 1_R\otimes(x\otimes\beta_\ast(y))\otimes 1_R
\end{align*}
for any $a,b,c\in X$, $\beta\in\Pi_1X(a,b)$ and $\gamma\in\Pi_1X(b,c)$, where $\beta_*$ is the map on the tensor factors of singular chains induced by the homeomorphism $\widetilde{X}_b\to \widetilde{X}_a$, $z\mapsto \beta z$.

For any $a,b\in X$, define $\widetilde{It}_X$ to be the comoposition of chain maps
\begin{equation}\label{equation: cube chain path space direct sum Pi_1X(a,b)}
    \mathsf{C}^{\square}(\mathsf{P}X)(a,b)
 \cong
 \bigoplus_{\beta\in\Pi_1X(a,b)}
 \mathsf{C}^{\square}(\mathsf{P}\widetilde{X}_a)([c_a],\beta)
 \to
 \mathsf{C^P}(X)(a,b),
\end{equation}
where the first map is defined in the same manner as \eqref{equation: cube chain path space direct sum pi_1}, and the second map is induced by $It_{\widetilde{X}_a}$ in \eqref{equation: It_X} on each direct summand.

\begin{remark}
The construction of $\mathsf{C^P}(X)$ and $It_X$ is motivated by an idea of Irie. 
Irie proposed a chain model for the free loop space (path space) and a version of the iterated integral map based on ``de Rham chains'' as a simplification of his construction in \cite{Irie}, which was established by the second-named author in \cite[Chapter 2]{Wang2023thesis}. 
This approach, unlike $\mathsf{C^P}$, uses the fundamental groupoid without referring to universal covering spaces.
\end{remark}

We extend $\widetilde{It}_X$ to an $A_\infty$-functor
\[
\widetilde{\mathcal{I}}_X=\{\widetilde{\mathcal{I}}_{X,k}\}_{k\ge1} \colon \mathsf{C}^{\square}(\mathsf{P}X) \to \mathsf{C^P}(X)
\]
with $\widetilde{\mathcal{I}}_{X,1}=\widetilde{It}_X$ as follows. Recall the cosimplicial space $[n]\mapsto\mathsf{P}^nX(a,b)$ in \eqref{equation: P^nX(a,b)} for $a,b\in X$. There is a sequence of canonical isomorphisms of chain complexes
\[
C(\mathsf{P}^nX(a,b))\cong \bigoplus_{\beta\in\Pi_1X(a,b)}C(\mathsf{P}^n\widetilde{X}_a([c_a],\beta)),\qquad n\ge0
\]
which is compatible with the cosimplicial structures, and induces an inclusion
\[ \bigoplus_{\beta\in\Pi_1X(a,b)}\mathrm{Tot}^{\prod}\big(\{C(\mathsf{P}^n\widetilde{X}_a([c_a],\beta))\}_{n\ge0}\big) \hookrightarrow \mathrm{Tot}^{\prod}\big(\{C(\mathsf{P}^nX(a,b))\}_{n\ge0}\big).
\]
For every $k\ge1$, the linear map $\mathcal{J}_{X,k}$ from Lemma~\ref{lemma: J_1 extends A infinity} factors through the above inclusion  with $a_0=a$, $a_k=b$. Then we define 
\[
\widetilde{\mathcal{I}}_{X,k}\coloneq\widetilde{AW}\circ\widetilde{ev}_*\circ\mathcal{J}_{X,k}\colon \bigotimes_{i=1}^k\mathsf{C}^{\square}(\mathsf{P}X)(a_{i-1},a_{i}) \to \mathsf{C^P}(X)(a_0,a_k),
\]
where the maps $\widetilde{\mathrm{ev}}_*$, $\widetilde{AW}$ are induced, on each direct summand, by
\begin{align*}
    \mathrm{ev}_* &\colon \mathrm{Tot}^{\prod}\big(\{C(\mathsf{P}^n\widetilde{X}_a([c_a],\beta))\}_{n\ge0}\big) \to \mathrm{Tot}^{\prod}\big(\{C(\{[c_a]\}\times(\widetilde{X}_a)^n\times\{\beta\})\}_{n\ge0}\big) \\
    \text{ and } AW&\colon \mathrm{Tot}^{\prod}\big(\{C(\{[c_a]\}\times(\widetilde{X}_a)^n\times\{\beta\})\}_{n\ge0}\big) \to \mathsf{Cobar}^{\prod}(C(\widetilde{X}_a))([c_a],\beta)
\end{align*}
as appearing in \eqref{equation: diagram J_X,1} and \eqref{equation: It_X original composition definition}, respectively.

We define the functor 
\[
\widetilde{\mathcal{G}}_X\colon\mathsf{Cobar}^{\boxtimes}(\mathcal{C}(X))\to\mathsf{C^P}(X)
\]
as follows. It acts as the identity on objects, and for $a,b\in X$ its action on morphisms is given by the composition of chain maps
\[
    \mathsf{Cobar}^{\boxtimes}(\mathcal{C}(X))(a,b)\cong\bigoplus_{\beta\in\Pi_1X(a,b)}\mathsf{Cobar}^{\boxtimes}(\mathcal{C}(\widetilde{X}_a))([c_a],\beta)
    \to \mathsf{C^P}(X),
\]
where the first map is defined in the same manner as \eqref{equation: Cobar boxtimes direct sum pi_1}, and the second map is induced by $\mathcal{G}_{\widetilde{X}_a}$ in \eqref{equation: G_X} on each direct summand.

We construct an $A_\infty$-homotopy $\widetilde{\mathcal{H}}_X$ between $\widetilde{\mathcal{F}}_X=\widetilde{\mathcal{I}}_X\circ\mathcal{A}_X$ and $\widetilde{\mathcal{G}}_X$ as follows. For any $k\ge1$ and any $k$-necklace in $X$
\[
\vec{\boldsymbol{\sigma}}\colon\mathbb{\Delta}^{\vec{\boldsymbol{d}}}\to X,\qquad \text{where }\vec{\boldsymbol{d}} = (\boldsymbol{d}_1, \dots, \boldsymbol{d}_k),\; \boldsymbol{d}_i = (d_{i,1}, \dots, d_{i,n_i}),
\]
as in \eqref{equation: multi necklace}, $\vec{\boldsymbol{\sigma}}$ uniquely lifts to a $k$-necklace in $\widetilde{X}_{\vec{\boldsymbol{\sigma}}(v_0)}$
\[
\widetilde{\vec{\boldsymbol{\sigma}}} \colon (\mathbb{\Delta}^{\vec{\boldsymbol{d}}},v_0)\to (\widetilde{X}_{\vec{\boldsymbol{\sigma}}(v_0)},[c_{\vec{\boldsymbol{\sigma}}(v_0)}])
\]
such that $\pi_{\vec{\boldsymbol{\sigma}}(v_0)}\circ\widetilde{\vec{\boldsymbol{\sigma}}}=\vec{\boldsymbol{\sigma}}$, where $v_0$ denotes the first vertex of $\mathbb{\Delta}^{\vec{\boldsymbol{d}}}$. Then we define
\[
\widetilde{\mathcal{H}}_X(\otimes_{i=1}^k\boxtimes_{j=1}^{n_i} s^{-1}\sigma_{i,j}) = \widetilde{\vec{\boldsymbol{\sigma}}}_*(\psi^{\vec{\boldsymbol{d}}})
\]
where $\psi^{\vec{\boldsymbol{d}}}\in\mathsf{Cobar}^{\prod}(C(\mathbb{\Delta}^{\vec{\boldsymbol{d}}}))(v_0,v_d)$ is provided in Lemma~\ref{lemma: homotopy Psi^d}.

By construction, $\widetilde{\mathcal{I}}_X$, $\widetilde{\mathcal{F}}_X$, $\widetilde{\mathcal{G}}_X$, and $\widetilde{\mathcal{H}}_X$ are natural in $X$. Theorem~\ref{theorem: A infinity map and homotopy with pi_1} follows.

\begin{remark}
There is a natural transformation 
\[
\pi\colon\mathsf{C^P}\Rightarrow\mathsf{Cobar}^{\prod}\circ C^{\mathbb{\Delta}}
\]
such that for any topological space $X$, $\pi_X$ is the identity on objects, and 
\[
\pi_X\colon\mathsf{C^P}(X)(a,b)
 \to \mathsf{Cobar}^{\prod}(C(X))(a,b),\qquad a,b\in X
\]
is induced by the universal covering map $\pi_a\colon(\widetilde{X}_a,[c_a])\to (X,a)$ on each direct summand. It is straightforward to check that 
$\mathcal{I}_X$, $\mathcal{F}_X$, $\mathcal{G}_X$ and $\mathcal{H}_X$ are the compositions of $\pi_X$ with $\widetilde{\mathcal{I}}_X$, $\widetilde{\mathcal{F}}_X$, $\widetilde{\mathcal{G}}_X$ and $\widetilde{\mathcal{H}}_X$, respectively. If $X$ is simply connected, then $\pi_X$ is an isomorphism of dg categories.
\end{remark}

It seems hard to formulate an analogue of Remark~\ref{remark: Chen Adams inverse simply connected} in the context here. Nonetheless, the following is true.

\begin{proposition}\label{proposition: widetilde G_X quasi equiv}
    For any topological space $X$, the functor $\widetilde{\mathcal{G}}_X$ is a quasi-equivalence.
\end{proposition}
\begin{proof}
    We know from Section~\ref{subsection: Chen Adams inverse simply connected} that $\mathcal{G}_X$ is a quasi-equivalence for simply connected $X$. The lemma follows by applying this to $\widetilde{X}_a$ in the construction of $\widetilde{\mathcal{G}}_X$.
\end{proof}

Theorem~\ref{theorem: A infinity map and homotopy with pi_1} and Proposition~\ref{proposition: widetilde G_X quasi equiv} imply that:
\begin{corollary}
    For any topological space $X$:
    \begin{enumerate}
        \item The functor \[\mathcal{A}_X \colon \mathsf{Cobar}^{\boxtimes}(\mathcal{C}(X))\to \mathsf{C}^{\square}(\mathsf{P}X)\] induced by Adams' map is a quasi-equivalence if and only if the $A_\infty$-functor 
    \[
    \widetilde{\mathcal{I}}_X \colon \mathsf{C}^{\square}(\mathsf{P}X)\to\mathsf{C^P}(X)
    \] 
    motivated by Chen's iterated integral map is an $A_\infty$-quasi-equivalence.
    \item (Given Theorem~\ref{theorem: Adams}) The $A_\infty$-functor $\widetilde{\mathcal{I}}_X$ is an $A_\infty$-quasi-equivalence.
    \end{enumerate}
\end{corollary}

\printbibliography

\end{document}